%% file: 21_06_JMIV_InverseScaleSpaceIterationsUsingFunctionalLifting.tex
\RequirePackage{fix-cm}
\documentclass[twocolumn]{svjour3}          % twocolumn

\smartqed  % flush right qed marks, e.g. at end of proof
\usepackage{graphicx}

\input{header.tex}

\journalname{Journal of Mathematical Imaging and Vision}

\usepackage{amsmath}

\begin{document}

\title{Inverse Scale Space Iterations for Non-Convex Variational Problems: The Continuous and Discrete Case
\thanks{The authors acknowledge support through DFG grant LE 4064/1-1 “Functional Lifting 2.0: Efficient Convexifications for Imaging and Vision” and NVIDIA Corporation.}
}

\titlerunning{Inverse Scale Space Iterations for Non-Convex Variational Problems}

\author{Danielle Bednarski \and Jan Lellmann}

\authorrunning{D. Bednarski, J. Lellmann} % if too long for running head

%\institute{D. Bednarski \at
%              Institute of Mathematics and Image Computing \\
%              University of Lübeck\\
%              Tel.: +49-451-3101-6125\\
%           J. Lellmann \at
%              Institute of Mathematics and Image Computing\\
%              University of Lübeck
%}

\institute{Danielle Bednarski \at
                   bednarski@mic.uni-luebeck.de
           \and
           Jan Lellmann \at
           lellmann@mic.uni-luebeck.de \\ ~ \\           
              Institute of Mathematics and Image Computing \\
              University of Lübeck \\
              Lübeck, Germany
}

\date{Received: date / Accepted: date}

\maketitle

\begin{abstract}
 
Non-linear filtering approaches allow to obtain decompositions of images with respect to a non-classical notion of scale, induced by the choice of a convex, absolutely one-homogeneous regularizer. The associated inverse scale space flow  can be obtained using the classical Bregman iteration with quadratic data term. We apply the Bregman iteration to lifted, i.e. higher-dimensional and convex, functionals in order to extend the scope of these approaches to functionals with arbitrary data term. We provide conditions for the subgradients of the regularizer -- in the continuous and discrete setting-- under which this lifted iteration reduces to the standard Bregman iteration. We show experimental results for the convex and non-convex case.

\end{abstract}

\input{Sections/introduction.tex}

\input{Sections/sa_relax.tex}
\input{Sections/lifted_breg.tex}

\input{Sections/lifted_breg_half.tex}

\input{Sections/numerical_discussion.tex}
\input{Sections/numerical_results.tex}

\input{Sections/conclusion.tex}

%\begin{acknowledgements}
%The authors acknowledge support through DFG grant LE 4064/1-1 “Functional Lifting 2.0: Efficient Convexifications for Imaging and Vision” and NVIDIA Corporation.
%\end{acknowledgements}

\section*{Conflict of interest}
The authors declare that they have no conflict of interest.

\bibliographystyle{splncs04}      % CURRENT
\bibliography{21_06_JMIV_InverseScaleSpaceIterationsUsingFunctionalLifting}   % name your BibTeX data base

\end{document}

%% file: header.tex
\usepackage[utf8]{inputenc}
\usepackage{graphicx}
\usepackage{amssymb}
\usepackage{amsmath}
\usepackage{bbm}
\usepackage{bbding}

\usepackage{hyperref}

\usepackage{tikz}
\usetikzlibrary{arrows}
\usepackage{pgfplots}
\usepgfplotslibrary{fillbetween}

\usepackage{caption}
\usepackage{booktabs}
\usepackage[font=small]{subcaption}
\usepackage{soul}
\usepackage{xcolor}

\makeatletter
\renewcommand{\@Opargbegintheorem}[4]{%
  #4\trivlist\item[\hskip\labelsep{#3#2\@thmcounterend}]}
\makeatother

\newcommand{\dbdetail}{}
\newcommand{\dbreformulated}{}
\newcommand{\dbold}{}

\newcommand{\dbe}{\color{black}}

\newcommand{\beq}{\begin{equation}}
\newcommand{\eeq}{\end{equation}}

\newcommand{\bal}{\begin{align}}
\newcommand{\eal}{\end{align}}

\newcommand{\beqn}{\begin{equation*}}
\newcommand{\eeqn}{\end{equation*}}

\newcommand{\bfl}{\begin{equation}\begin{aligned}}
\newcommand{\efl}{\end{aligned}\end{equation}}

\newcommand{\bfln}{\begin{equation*}\begin{aligned}}
\newcommand{\efln}{\end{aligned}\end{equation*}}

\newcommand{\into}{\int_{\Omega}}
\newcommand{\intog}{\int_{\Omega\times\Gamma}}

\newcommand{\otg}{\Omega\times\Gamma}

\newcommand{\dx}{\,\text{d}x}
\newcommand{\dt}{\,\text{d}t}
\newcommand{\dxt}{\,\text{d}(x,t)}

\newcommand{\dH}{\text{d}\mathcal{H}}

\newcommand{\bs}{\boldsymbol}

\newcommand{\Hd}{\mathcal{H}}
\newcommand{\F}{\mathcal{F}}
\newcommand{\K}{\mathcal{K}}

\newcommand{\R}{\mathbb{R}}
\newcommand{\Rc}{\overline{\mathbb{R}}}

\newcommand{\TV}{\text{TV}}
\newcommand{\BV}{\text{BV}}

\DeclareMathOperator{\sgn}{sgn}

\newcommand{\ls}{\left\{}
\newcommand{\rs}{\right\}}

\renewcommand{\ll}{\left\langle}
\newcommand{\rr}{\right\rangle}

\DeclareMathOperator{\mydiv}{div}
\newcommand{\mydivx}{\mydiv_x}

\DeclareMathOperator{\mycl}{cl}

\definecolor{oceangreen}{RGB}{0,57,74}
\definecolor{oceangreen-90}{RGB}{0,73,90}
\definecolor{oceangreen-80}{RGB}{0,90,107}
\definecolor{oceangreen-70}{RGB}{17,107,123}
\definecolor{oceangreen-60}{RGB}{60,126,142}
\definecolor{oceangreen-50}{RGB}{93,146,160}
\definecolor{oceangreen-40}{RGB}{126,168,178}
\definecolor{oceangreen-30}{RGB}{157,188,198}
\definecolor{oceangreen-20}{RGB}{189,210,216}
\definecolor{oceangreen-10}{RGB}{222,233,236}
\definecolor{red-lightest}{RGB}{233,191,173}
\definecolor{red-light}{RGB}{227,32,50}
\definecolor{red-medium}{RGB}{182,22,32}
\definecolor{red-dark}{RGB}{126,20,24}
\definecolor{orange-lightest}{RGB}{240,205,178}
\definecolor{orange-light}{RGB}{236,117,4}
\definecolor{orange-medium}{RGB}{202,81,25}
\definecolor{orange-dark}{RGB}{129,53,18}
\definecolor{yellow-lightest}{RGB}{236,217,186}
\definecolor{yellow-light}{RGB}{251,186,0}
\definecolor{yellow-medium}{RGB}{191,134,21}
\definecolor{yellow-dark}{RGB}{117,83,17}
\definecolor{green-lightest}{RGB}{219,215,187}
\definecolor{green-light}{RGB}{149,187,12}
\definecolor{green-medium}{RGB}{126,133,37}
\definecolor{green-dark}{RGB}{50,89,74}
\definecolor{turquoise-lightest}{RGB}{205,219,216}
\definecolor{turquoise-light}{RGB}{59,178,160}
\definecolor{turquoise-medium}{RGB}{36,143,133}
\definecolor{turquoise-dark}{RGB}{0,90,91}
\definecolor{ocean-lightest}{RGB}{198,220,226}
\definecolor{ocean-light}{RGB}{0,174,198}
\definecolor{ocean-medium}{RGB}{0,145,167}
\definecolor{ocean-dark}{RGB}{0,97,122}
\definecolor{cyan-lightest}{RGB}{195,217,229}
\definecolor{cyan-light}{RGB}{60,169,213}
\definecolor{cyan-medium}{RGB}{0,131,173}
\definecolor{cyan-dark}{RGB}{0,87,119}
\definecolor{blue-lightest}{RGB}{195,217,236}
\definecolor{blue-light}{RGB}{111,165,206}
\definecolor{blue-medium}{RGB}{0,105,163}
\definecolor{blue-dark}{RGB}{0,70,114}
\definecolor{gray-cold-lightest}{RGB}{215,223,228}
\definecolor{gray-cold-light}{RGB}{191,203,213}
\definecolor{gray-cold-medium}{RGB}{140,157,171}
\definecolor{gray-cold-dark}{RGB}{87,105,120}
\definecolor{gray-neutral-lightest}{RGB}{226,228,228}
\definecolor{gray-neutral-light}{RGB}{208,208,210}
\definecolor{gray-neutral-medium}{RGB}{156,157,160}
\definecolor{gray-neutral-dark}{RGB}{100,100,102}
\definecolor{gray-warm-lightest}{RGB}{237,236,227}
\definecolor{gray-warm-light}{RGB}{218,214,203}
\definecolor{gray-warm-medium}{RGB}{162,159,145}
\definecolor{gray-warm-dark}{RGB}{118,115,105}

\def\zlima{0.1}
\def\zlimb{0.9}

%% file: Sections/introduction.tex
\section{Motivation and Introduction}
\label{sec:introduction}
\input{Sections/figure_profile.tex}

In modern image processing tasks, \emph{variational problems} constitute an important tool \mbox{\cite{book_aubert,book_scherzer}}. They are used in a variety of applications such as denoising \mbox{\cite{rof}}, segmentation \mbox{\cite{chan_vese}}, and depth estimation \mbox{\cite{depth1,depth2}}. In this work, we consider variational image processing problems with energies of the form
\begin{equation}
F(u) := \underbrace{\into \rho(x,u(x)) \dx}_{H(u)} + \underbrace{\into \eta(\nabla u(x)) \dx}_{J(u)} ,
\label{eq:problem}
\end{equation}
where the integrand $\eta : \R^d \mapsto \R$ of the \emph{regularizer} is non-negative and convex, and the integrand $\rho : \Omega \times \Gamma \mapsto \Rc$ of the \emph{data term} $H$ is proper, non-negative and possibly non-convex with respect to~$u$. We assume that the domain $\Omega \subset \R^d$ is open and bounded and that the range $\Gamma \subset \R$ is compact. After discretization, we refer to $\Gamma$ as the \emph{label space} in analogy to multi-label problems with discrete range \cite{ishikawa}.

We are mainly concerned with three distinct problem classes.  Whenever we are working with the \emph{total variation} regularizer \cite{book_ambrosio}
\begin{equation}
J(u) = \TV(u),
\end{equation}
we use the abbreviation TV-\eqref{eq:problem}. This problem class is of special interest, since it allows to use a ``sublabel-accurate'' discretization \cite{sublabel_cvpr}.

If the data term is quadratic,
\begin{equation}
\rho(x,u(x)) = \frac{\lambda}{2}(u(x)-f(x))^2\label{eq:rof_data},
\end{equation}
for some input $f$ and $\lambda > 0$ we use the abbreviation ROF-\eqref{eq:problem}. This Rudin-Osher-Fatemi problem is the original use case for the Bregman iteration \cite{osher2005iterative}.  For the quadratic data term~\eqref{eq:rof_data} and  an arbitrary convex, absolutely one-homogeneous regularizer~$\eta$, we write OH-\eqref{eq:problem}. This problem class has been extensively studied  by the inverse scale space flow community. For more details on the inverse scale space flow we refer to the next section.

In this work, we aim to combine the lifting approach \cite{lifting_global_solutions,sublabel_cvpr,sublabel_discretization}, which allows to solve problems a with non-convex data term in a higher-dimensional space in a convex fashion, with the Bregman iteration \cite{osher2005iterative}, which recovers a scale space of solutions. This provides a natural and practical extension of the Bregman iteration to nonconvex data terms. In the following, we briefly review each of these concepts.

\subsection{Inverse Scale Space Flow}

Consider the so-called \emph{inverse scale space flow} \cite{osher2005iterative,burger2006nonlinear,burger2016spectral} equation
\bfl
\partial_s p(s) =  f -u(s,\cdot), \, \, p(s) \in \partial J(u(s,\cdot)), \, \, p(0) =  0,
\label{eq:iss}
\efl
where $J$ is assumed to be convex and absolutely one-homogeneous.
The evolution $u:[0,T]\times \Omega \to \R$ in \eqref{eq:iss} starts at $u(0,\cdot)=\text{mean}(f)$ and $p(s)$ is forced to lie in the subdifferential of $J$.

One can show \cite{gilboa2014total,burger2015spectral} that  during the evolution non-linear \emph{eigenfunctions} of the regularizer with increasingly larger eigenvalues are added to the flow, where a nonlinear \emph{\emph{eigenfunction}} $u$ of $J$ for an eigenvalue $\lambda$ is understood as a solution of the inclusion
\bfl
\lambda u \in \partial J(u).
\efl 
Typically, and in particular for the total variation regularization $J=\TV$,
the flow $u(s,\cdot)$ incorporates details of the input image at progressively finer scales as~$s$ increases. Large-scale details can be understood as the \emph{structure} of the image and fine-scale details as the \emph{texture}. Stopping the flow at a suitable time returns a denoised image, whereas
for $s\to\infty$ the flow converges to the input image. 

By considering the derivative $u_s$, one can define a non-linear decomposition of the input $f$ \cite{burger2015spectral,gilboa2016nonlinear}: First a transformation to the spectral domain of the regularizer is defined. After the transformation, the input data is represented as a linear combination of generalized eigenfunctions of the regularizer. The use of filters in the spectral domain followed by a reconstruction to the spatial domain leads to a high-quality decomposition of the input image into different scales \cite{hait2018spectral}. Similar ideas have been developed for
variational models of the form OH-\eqref{eq:problem} and \textit{gradient flow} formulations   \cite{burger2006nonlinear,benning2012ground,gilboa2013spectral,gilboa2014total,burger2016spectral,gilboa2017semi}. 

As we will see in the next section, the first part of the flow equation {\eqref{eq:iss}} directly relates to the derivative of the quadratic data term {\eqref{eq:rof_data}}.
To the best of our knowledge, the question of how to define similar scale space transformations and filters for the solutions of variational problems with \emph{arbitrary} data terms has not yet been studied.

\input{Sections/figure_calibration.tex}

\subsection{Bregman Iteration}

For problems in the class OH-\eqref{eq:problem}, the inverse scale space flow can be understood \cite{burger2006nonlinear} as a continuous limit of the \emph{Bregman iteration} \cite{osher2005iterative}.
The Bregman iteration uses the Bregman divergence first introduced in \cite{bregman1967relaxation}. 
For both the data term~$H$ and regularizer $J$ being non-negative and convex (!), the Bregman iteration is defined as: 

 \noindent              \fbox{%
                \begin{minipage}[t]{0.95\linewidth}
                \textbf{Algorithm 1: Bregman iteration}
                        \begin{itemize}
                        \item[] Initialize $p_0 = 0$ and repeat for $k = 1,2,...$
                                                \begin{flalign}
                                u_k & \in \arg\min_{u} \{ H(u) + J(u) - \langle p_{k-1}, u \rangle \} , \label{eq:breg_uk}\\
                                p_{k} & \in \partial J(u_k) \label{eq:breg_qk}. 
                        \end{flalign}                   
                        \end{itemize}                                           
                \end{minipage}} \\

In case of the ROF-\eqref{eq:problem} problem, the subgradient $p_k$ can be chosen explicitly as $p_k = p_{k-1} - \lambda (u_k - f)$. Rearranging this equation as $(p_k - p_{k-1})/\lambda =-(u_k-f)$ shows that it is simply a time step for \eqref{eq:iss}. An extensive analysis of the iteration including well-definedness of the iterates and convergence results can be found in~\cite{osher2005iterative}.

Further extensions include the \emph{split Bregman method} for $\ell_1$-regularized problems \cite{goldstein2009split} and the \emph{linearized Bregman iteration} for compressive sensing and sparse denoising \cite{cai2009linearized,osher2011fast}. However, applying the Bregman iteration to variational problems with non-convex data term $H$ is not trivial since the well-definedness of the iterations, the use of subgradients, as well as the convergence results in \cite{osher2005iterative} rely on the convexity of the data term. 
In \cite{hoeltgen2015bregman}, the Bregman iteration was used to solve a non-convex \emph{optical flow} problem, however, the approach relies on an iterative reduction to a convex problem using first-order Taylor approximations. 

In this work, we aim to apply the Bregman iteration to energies with a non-convex data term such as the non-convex stereo matching problem (Fig.~\ref{fig:stereo_profile} and Fig.~\ref{fig:results2}). In order to do so, we follow a \emph{lifting approach} \cite{lifting_global_solutions,sublabel_cvpr,sublabel_discretization}: Instead of minimizing the non-convex problem 
\beq
\inf_{u\in U} \{ H(u)+J(u) \}\label{eq:prob-ulf}
\eeq
over some suitable (discrete or function) space $U$, 
we solve a \emph{lifted problem}
\bfl
\inf_{v\in \mathcal{U}} \{ \mathcal{H}(v)+\mathcal{J}(v) \}\label{eq:prob-lf}
\efl
over a larger space $\mathcal{U}$ but with \emph{convex} energies $\mathcal{H}, \mathcal{J}$ in a way that allows to recover solutions $u$ of the original problem \eqref{eq:prob-ulf} from solutions $v$ of the lifted problem \eqref{eq:prob-lf}. The Bregman iteration can then be performed on the -- now convex -- lifted problem \eqref{eq:prob-lf}:

 \noindent              \fbox{%
                \begin{minipage}[t]{0.95\linewidth}
                \textbf{Algorithm 2: Lifted Bregman iteration}
                        \begin{itemize}
                        \item[] Initialize $p_0 = 0$ and repeat for $k = 1,2,...$
                                                \begin{flalign}
                                v_k & \in \arg\min_{v\in\mathcal{U}} \{ \mathcal{H}(v) + \mathcal{J}(v) - \langle p_{k-1}, v \rangle \} , \label{eq:breg_uk_lifted}\\
                                p_{k} & \in \partial \mathcal{J}(v_k) \label{eq:breg_qk_lifted}. 
                        \end{flalign}                   
                        \end{itemize}                                           
                \end{minipage}} \\
                
This allows to extend the Bregman iteration to non-convex data terms. Since the Bregman iteration crucially depends on the choice of subgradients, it is not evident how the original (Alg.~1) and lifted (Alg.~2) Bregman iteration relate to each other in case of the prototypical ROF-\eqref{eq:problem} problem and we analyse this question in the continuous and discrete setting.

\subsection{Outline and Contribution}

This work is an extension of the conference report \cite{bednarski2021inverse}. Compared to the report, we expand our theoretical analysis of the lifted Bregman iteration to the fully continuous setting and present analogous statements about the equivalence of the original and lifted Bregman iteration under certain assumptions. Additional numerical experiments demonstrate that eigenfunctions of the TV regularizer appear according to the size of their eigenvalues at different steps of the iteration -- this also holds true for the non-convex and non-linear stereo matching problem.

In section 2, we summarize the lifting approach for problems of the form TV-\eqref{eq:problem} both in continuous (function space) and the discretized (sublabel-accurate)\ formulations. We derive conditions under which the original and lifted Bregman iteration are equivalent in the continuous (section 3) and in the discretized (section 4--5) setting. The conditions in the discretized setting are in particular met by the anisotropic TV. In section 6, we validate our findings experimentally by comparing the original and lifted iteration on the convex ROF-\eqref{eq:problem} problem and present numerical results on the non-convex stereo matching problem.

\subsection{Notation}
We denote the \emph{extended real line} by $\Rc := \R \cup \{ \pm \infty\}$. Given a function $f:\R^n\mapsto\Rc$, the conjugate $f^*:\R^n\mapsto\Rc$ is defined as \cite[Ch.~11]{book_rock_variational}
\begin{equation}
f^*(u^*) := \sup_{u\in\R^n} \{ \langle u^*, u \rangle -f(u) \}.
\end{equation}
If $f$ has a proper convex hull, both the conjugate and biconjugate are proper, lower semi-continuous and convex \cite[Ch.~11]{book_rock_variational}. 
The \emph{indicator function} of a set $C$ is defined as 
\begin{equation} \delta_C(x) := \begin{cases} 0, &\text{ if } x \in C, \\ \infty, &\text{ else.} \end{cases} 
\end{equation}
The Fenchel conjugate can similarly defined on general normed spaces by taking $u^\ast$ from the dual space \cite[Def.~I.4.1]{book_ekeland}.

Whenever $\bs{u}$ denotes a vector, we use subscripts $\bs{u}_k$ to indicate an iteration or sequence, and superscripts $\bs{u}^k$ to indicate the $k$-th value of the vector. We use calligraphic letters to denote lifted  energies in the continuous setting (e.g., $\F, \K, \mathcal{TV}$) and bold letters to denote lifted energies in the discrete setting (e.g., $\bs{F}, \bs{K}, \bs{TV}$).

The total variation regularizer is defined as
\begin{align}
TV(u) &:= \sup_{\substack{\psi \in C_c^\infty(\Omega; \, \R^d), \\  \|\psi (x) \|_2 \leq 1}}  \ls \into \langle u, - \text{div } \psi \rangle \dx \rs.
\label{eq:tv_def}
\end{align}
By $\BV(\Omega;\Gamma)$ we denote the set of functions $u$ that are of bounded variation, i.e. for which $\TV(u)<\infty$.

%% file: Sections/figure_profile.tex
\begin{figure*}[ht]
    \begin{subfigure}{.27\linewidth}
        \centering
        \includegraphics[width = \linewidth]{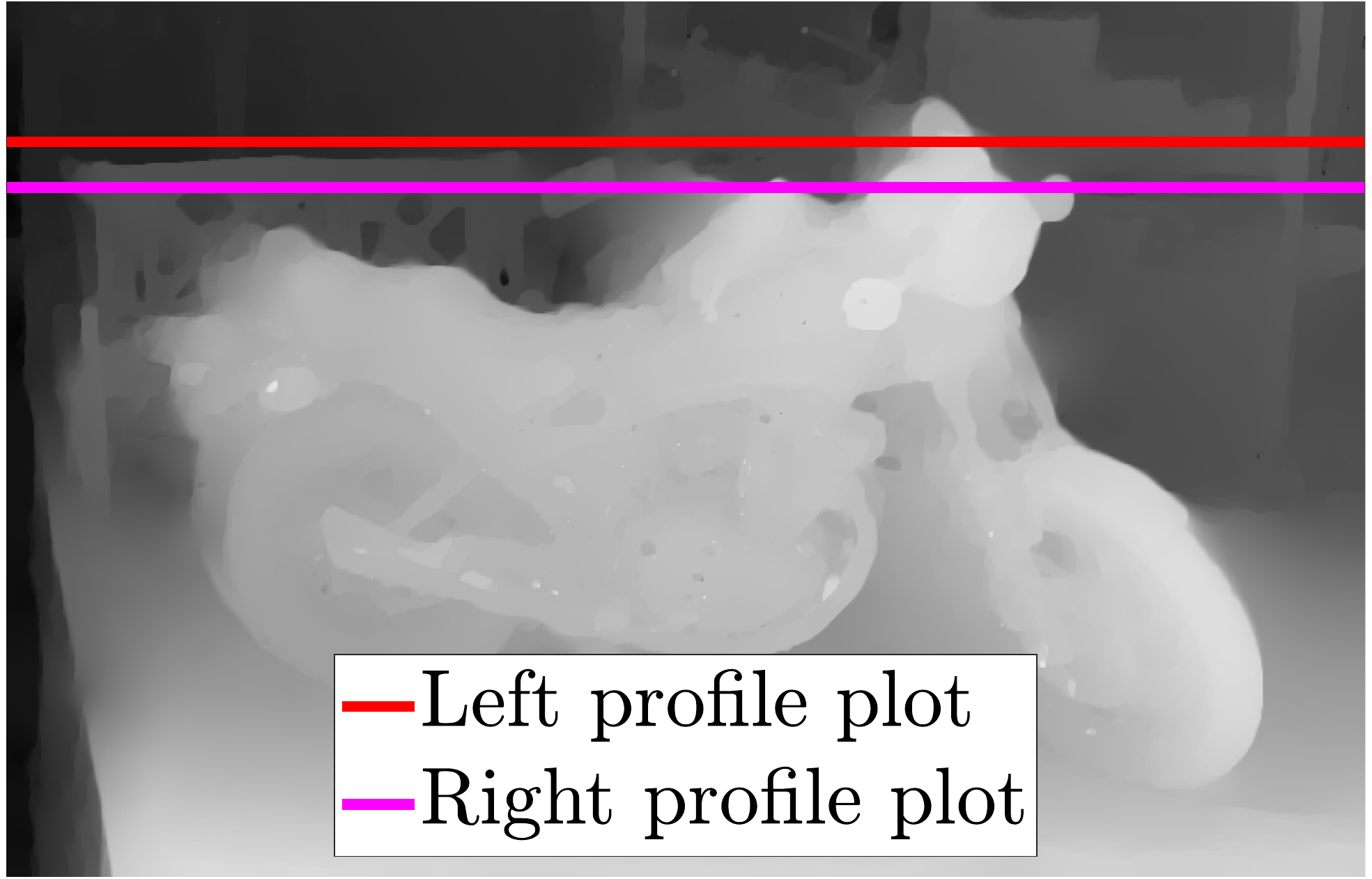}
    \end{subfigure}%
    \hspace{0.4cm}% Space between image A and B
    \begin{subfigure}{.30\linewidth}
        \centering
        \includegraphics[width = \linewidth]{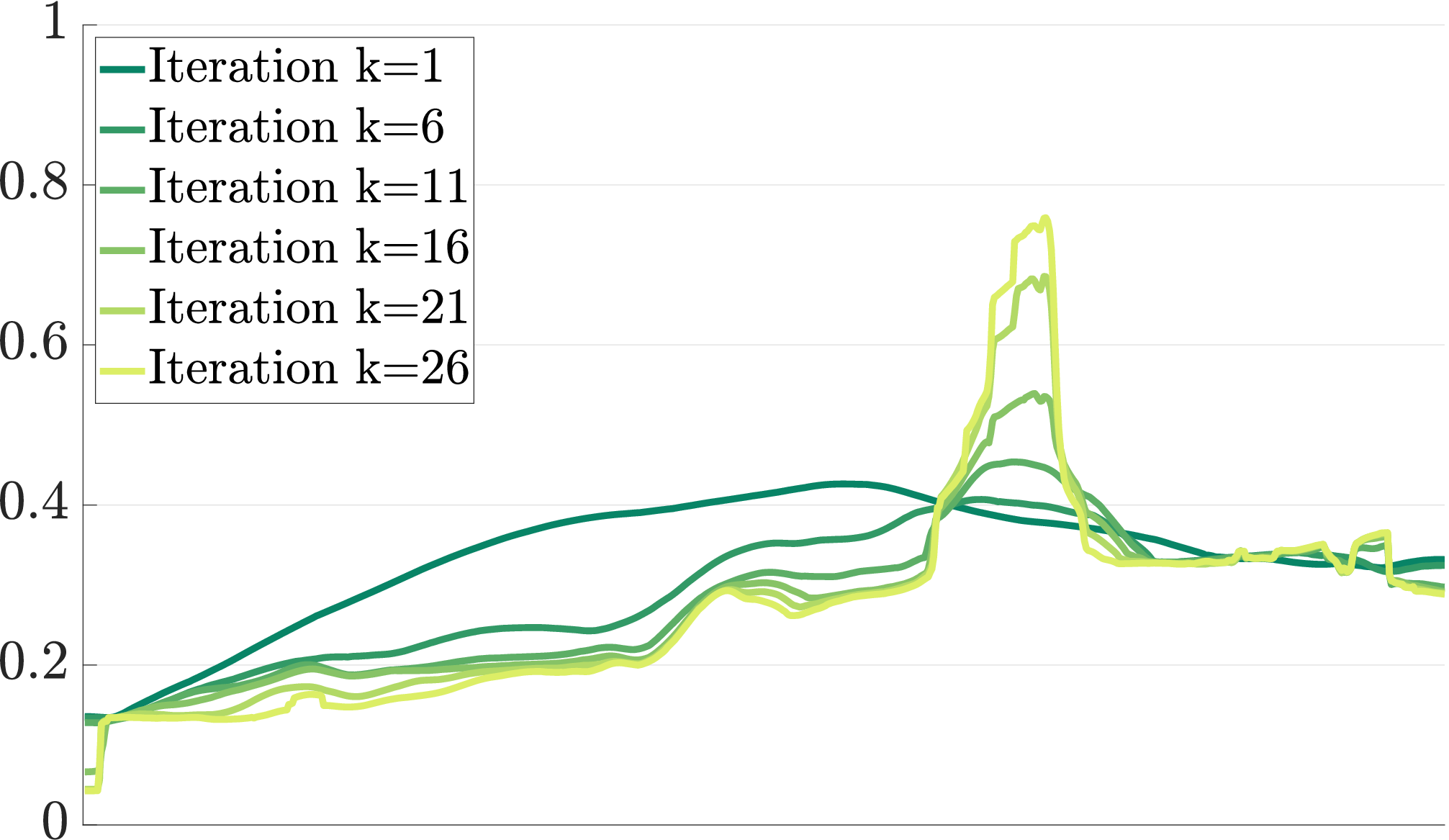}
    \end{subfigure}%
    \hspace{0.4cm}% Space between image A and B
    \begin{subfigure}{.30\linewidth}
        \centering
        \includegraphics[width = \linewidth]{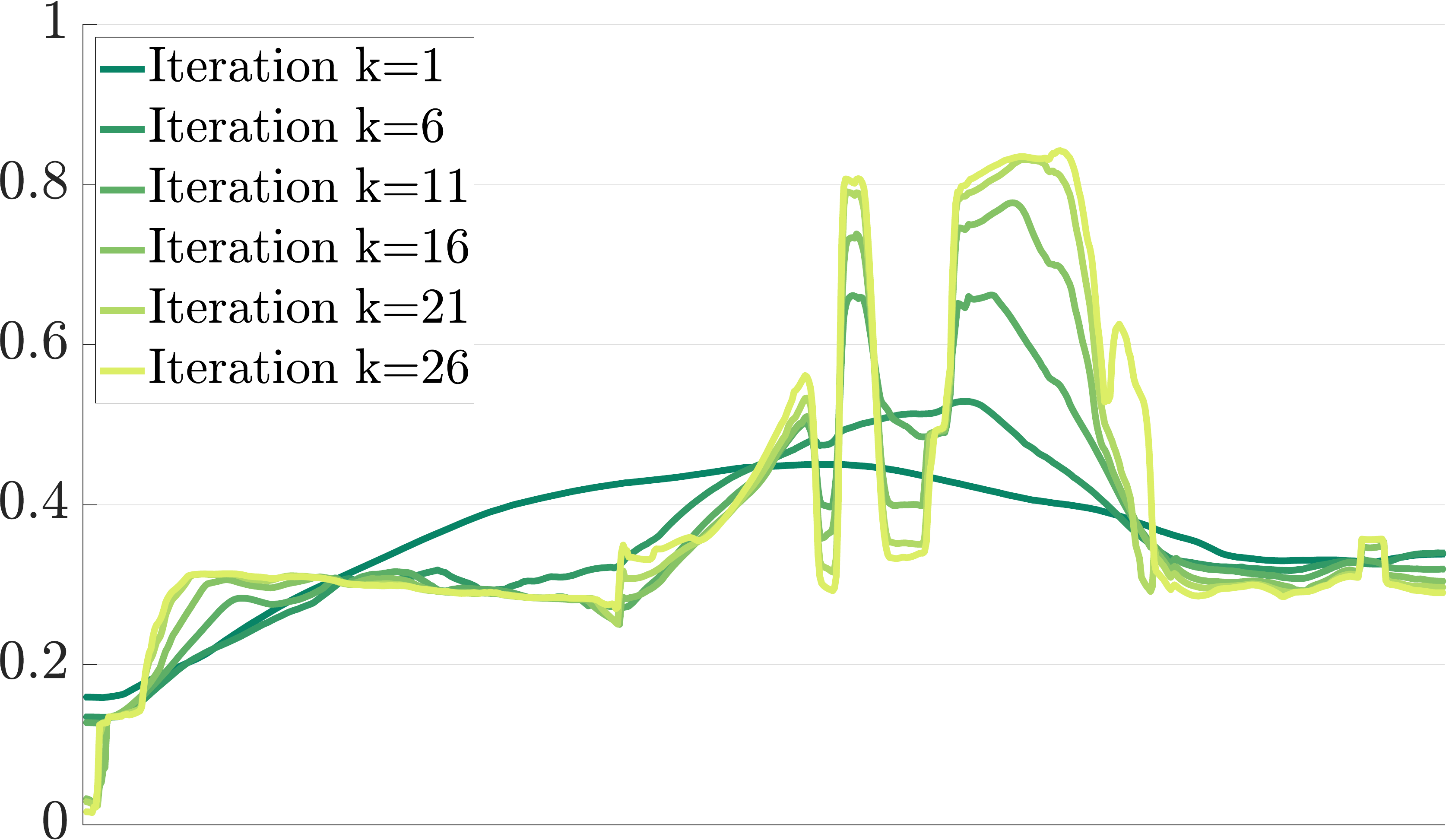}
        \end{subfigure}%        
     \caption{\textbf{Scale-space of solutions for non-convex depth estimation.} 
We apply the sublabel-accurate lifting approach \cite{sublabel_cvpr} to the non-convex problem of depth estimation in order to obtain a convex problem to which the  Bregman iteration~\cite{osher2005iterative} can be applied. In addition to the final depth map \textbf{(left)}, the Bregman iteration generates a scale space of solutions with increasing spatial detail, as can be seen from the two horizontal sections \textbf{(center, right).} 
}
        \label{fig:stereo_profile}
\end{figure*}

%% file: Sections/figure_calibration.tex
 \begin{figure*}{H}
     \begin{subfigure}[b]{0.28\textwidth}
          \centering
          \resizebox{\linewidth}{!}{\begin{tikzpicture}
    \pgfmathsetmacro\xmin{-0.05}
    \pgfmathsetmacro\xmax{1.05}
    \pgfmathsetmacro\ymin{\zlima}
    \pgfmathsetmacro\ymax{\zlimb}
    \begin{axis}[
        width=1\textwidth,
        axis lines=left,
        xlabel={$\Omega$}, xlabel near ticks, 
        ylabel={$\Gamma$}, ylabel near ticks,
        xmin=\xmin,xmax=\xmax,
        ymin=\ymin,ymax=\ymax,
        xtick=\empty,ytick=\empty,
        clip=false]
    \addplot graphics [
        xmin=0,xmax=1,
        ymin=0.11,ymax=\zlimb,
    ] {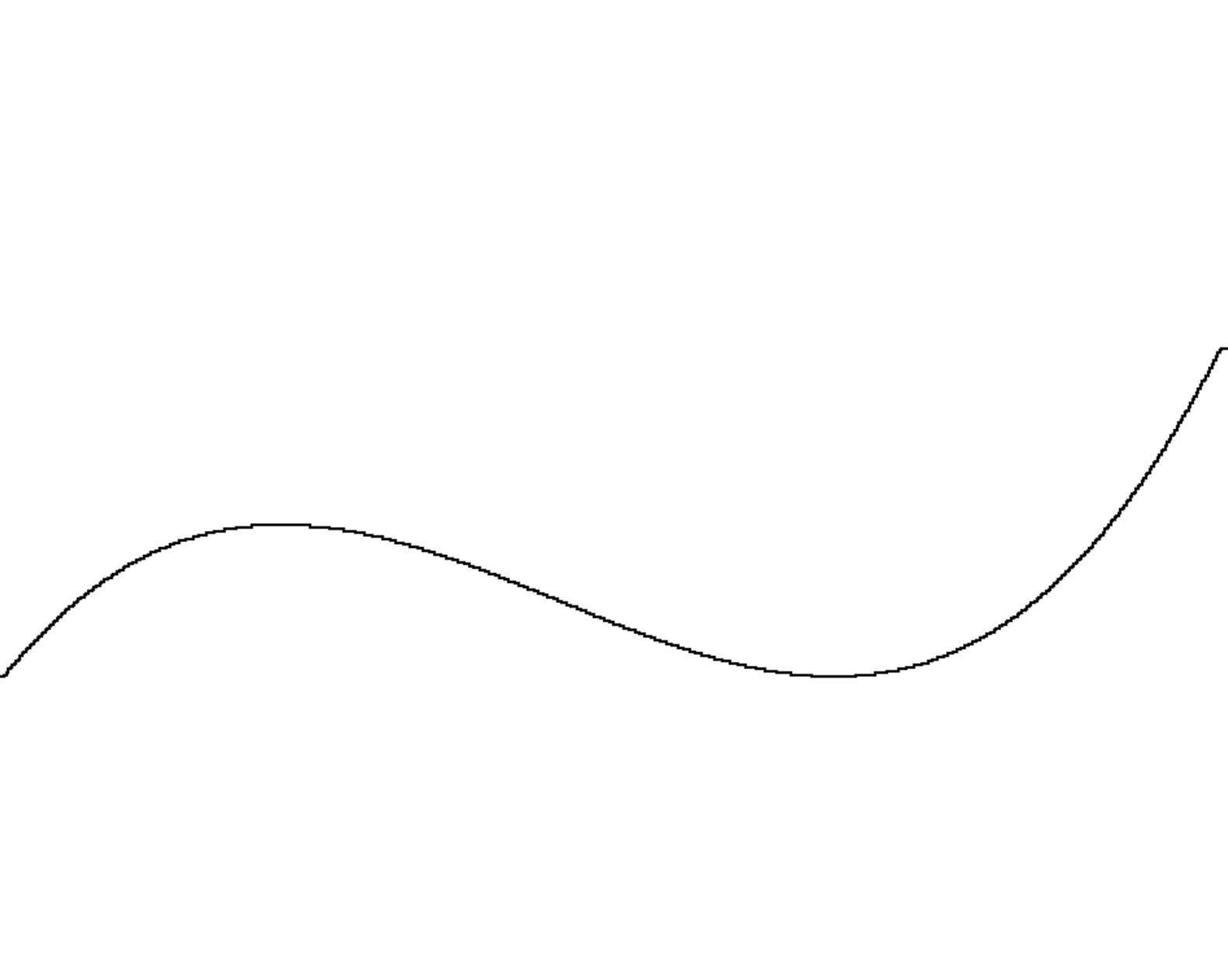};
    %\addplot coordinates {(0,0.1) (1,0.1)};
    \node at (100,600) {$u(x)$};
    \end{axis}
\end{tikzpicture}}  
     \end{subfigure} \hspace{1cm}
     \begin{subfigure}[b]{0.28\textwidth}
          \centering
          \resizebox{\linewidth}{!}{\begin{tikzpicture}
    \pgfmathsetmacro\xmin{-0.05}
    \pgfmathsetmacro\xmax{1.05}
    \pgfmathsetmacro\ymin{\zlima}
    \pgfmathsetmacro\ymax{\zlimb}
    \begin{axis}[
        width=1\textwidth,
        axis lines=left,
        xlabel={$\Omega$},  xlabel near ticks, 
        ylabel={$\Gamma$}, ylabel near ticks,
        xmin=\xmin,xmax=\xmax,
        ymin=\ymin,ymax=\ymax,
        xtick=\empty,ytick=\empty,
        clip=false]
    \addplot graphics [
        xmin=0,xmax=1,
        ymin=\zlima,ymax=\zlimb,
    ] {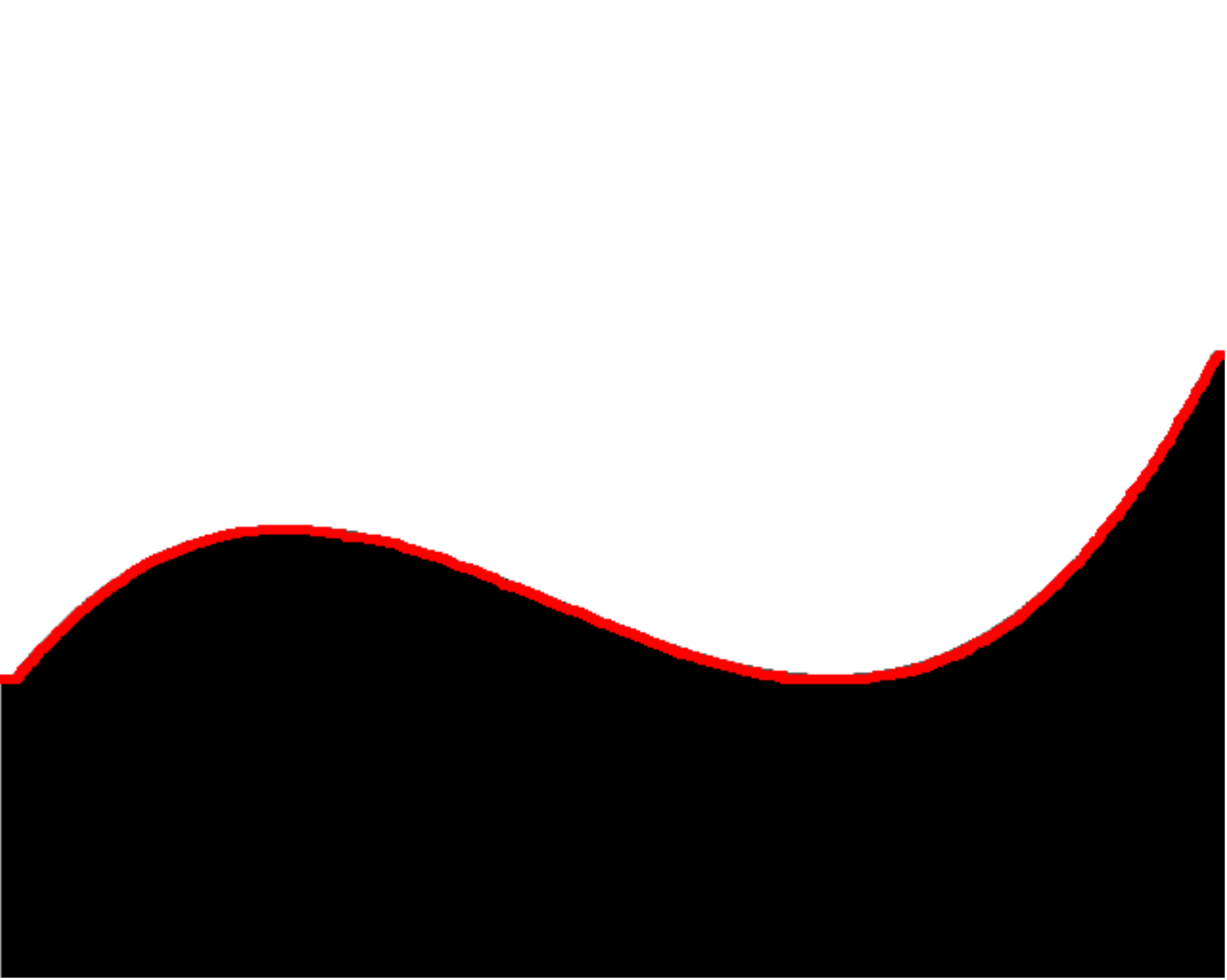};
    \addplot [
        quiver={u=\thisrow{u},v=\thisrow{v}},
        -stealth,
        green-medium,
    ] table [row sep=\\]{
        x y u v \\
        0.10 0.62 0.05 -0.1 \\
        0.18 0.65 0.05 -0.1 \\
        0.26 0.66 0.05 -0.1 \\
        0.34 0.65 0.05 -0.1 \\
        0.42 0.64 0.05 -0.1 \\
    };
    \draw[->, red](26,370)--(26,250);
    \node at (38, 280) {$\textcolor{red}{\nu_{\Gamma_u}}$};
    \node at (30,600) {$\textcolor{green-medium}{\phi(x,t)}$};
    \node at (70,650) {$\textcolor{black}{\bs{1}_u(x,t) = 0}$};
    \node at (70,150) {$\textcolor{white}{\bs{1}_u(x,t) = 1}$};
    \node at (95,500) {$\textcolor{red}{\Gamma_u}$};   
    \end{axis}
\end{tikzpicture}}  
     \end{subfigure} \hspace{1cm}
     \begin{subfigure}[b]{0.28\textwidth}
          \centering
          \resizebox{\linewidth}{!}{\begin{tikzpicture}
    \pgfmathsetmacro\xmin{-0.05}
    \pgfmathsetmacro\xmax{1.05}
    \pgfmathsetmacro\ymin{\zlima}
    \pgfmathsetmacro\ymax{\zlimb}
    \begin{axis}[
        width=1\textwidth,
        axis lines=left,
        xlabel={$\Omega$},  xlabel near ticks,
        ylabel={$\Gamma$}, ylabel near ticks,
        xmin=\xmin,xmax=\xmax,
        ymin=\ymin,ymax=\ymax,
        xtick=\empty,ytick=\empty,
        clip=false]
    \addplot graphics [
        xmin=0,xmax=1,
        ymin=\zlima,ymax=\zlimb,
    ] {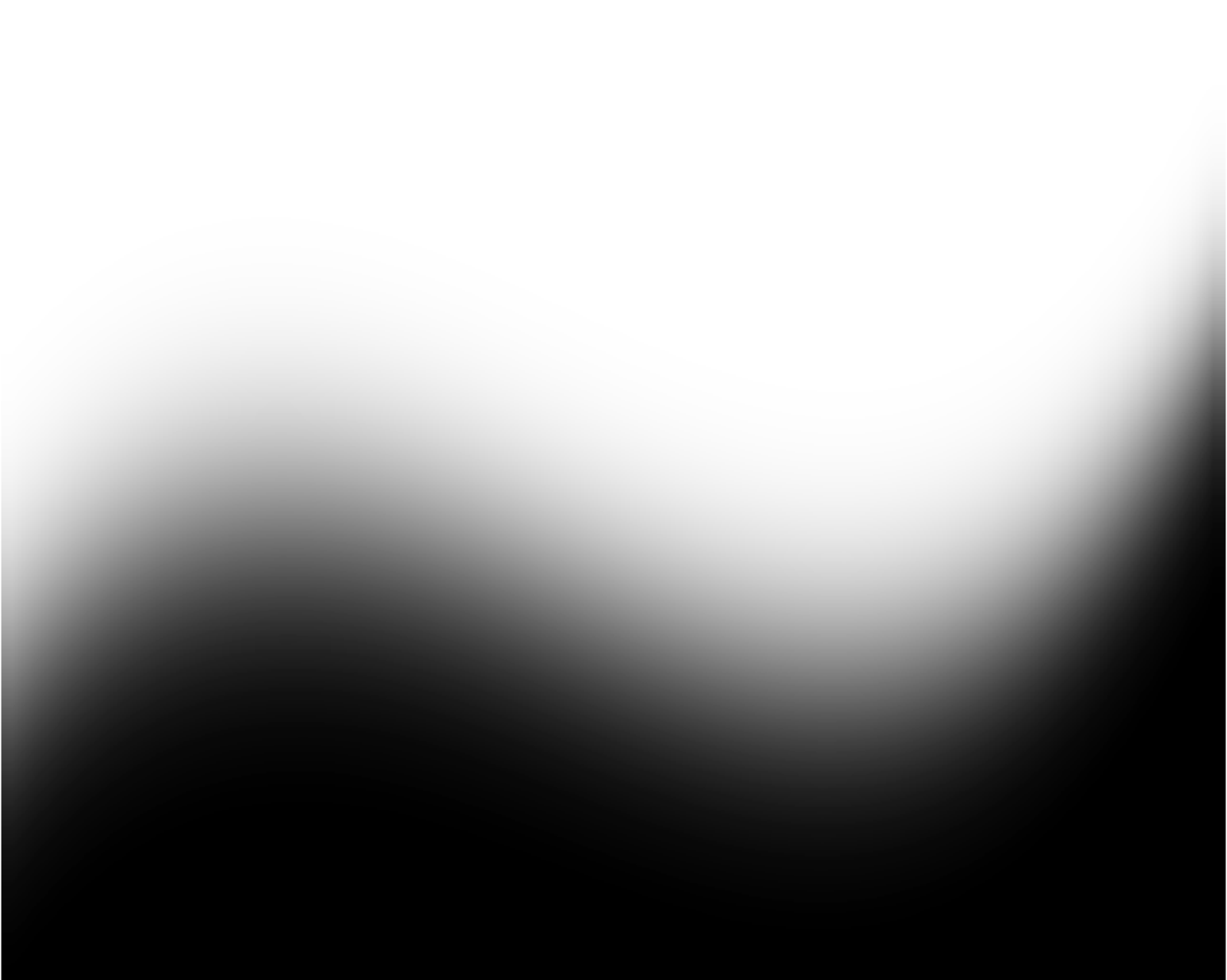};
    \addplot [
        quiver={u=\thisrow{u},v=\thisrow{v}},
        -stealth,
        green-medium,
    ] table [row sep=\\]{
        x y u v \\
        0.10 0.62 0.05 -0.1 \\
        0.18 0.65 0.05 -0.1 \\
        0.26 0.66 0.05 -0.1 \\
        0.34 0.65 0.05 -0.1 \\
        0.42 0.64 0.05 -0.1 \\
    };
    \node at (30,600) {$\textcolor{green-medium}{\phi(x,t)}$};
    \node at (70,650) {$\textcolor{black}{v(x,t) = 0}$};
    \node at (70,150) {$\textcolor{white}{v(x,t) = 1}$};
    \end{axis}
\end{tikzpicture}}
     \end{subfigure}
     \caption{ \textbf{Calibration based lifting.} Function $u$ in original solution space \textbf{(left)}. The idea in calibration based lifting is to represent functions $u$ by the higher dimensional indicator function of their subgraph $\bs{1}_u$. The variational problem is then rewritten as a flux of vector fields $\phi$ through the complete graph $\Gamma_u$; the latter is the measure theoretic boundary of the subgraph. Here, $\nu_{\Gamma_u}$ denotes the inner (downwards-pointing) unit normal \textbf{(middle)}. By enlarging the feasible set to the convex hull of the subgraph indicator functions, which also allows diffuse solutions \textbf{(right)}, one obtains an overall convex problem. }
     \label{fig:calibration}
 \end{figure*}
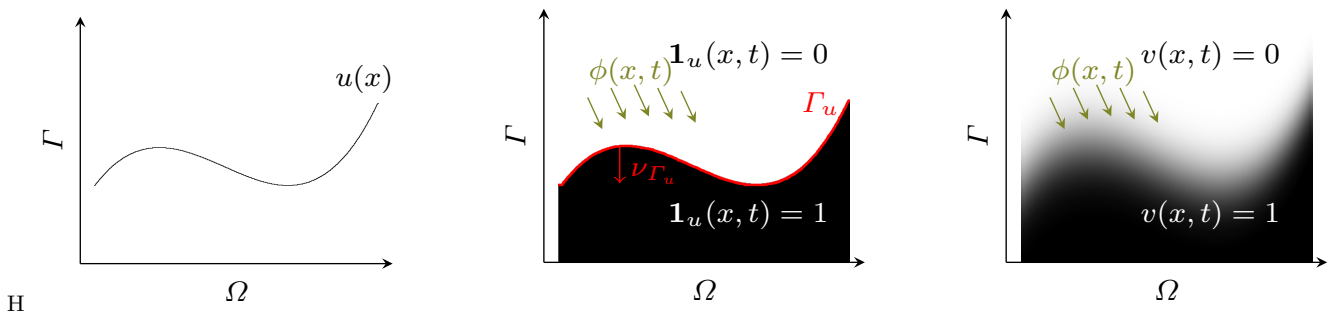

%% file: Sections/sa_relax.tex
\section{Lifting Approach}

\subsection{Continuous Setting}\label{ss:contset} %

Our work is based on methods for scalar but continuous range $\Gamma$ with first-order regularization in the spatially continuous setting 
\cite{alberti2003calibration,chambolle2001convex,lifting_tv,lifting_global_solutions}: The feasible set of the scalar-valued functions $u\in W^{1,1}(\Omega;\Gamma)$ is embedded into the convex set of functions which are of bounded variation on every compact subset of $\Omega\times\R$ (i.e.  $v\in \BV_{\text{loc}}(\Omega\times\R; [0,1])$) by associating each function $u$ with the
characteristic function of its subgraph, 
i.e.,
\begin{align} 
1_u(x,z) :=
\begin{cases}
1, & u(x) > z, \\
0, & \text{otherwise.} \\
\end{cases}
\end{align}
To extend the energy $F$ in \eqref{eq:problem} for $\Gamma = \R$ onto this larger space, a \emph{lifted} convex functional $\mathcal{F}: \BV_{\text{loc}}(\Omega\times\R; [0,1])\to\R$ based on the distributional derivative $Dv$ is defined \cite{alberti2003calibration}:
\bfl
\mathcal{F}(v) := \sup_{\phi \in \mathcal{K}} \ls \int_{\Omega \times \Gamma} \langle \phi, Dv \rangle \rs,
\label{eq:lifted_F}
\efl
where the admissible dual vector fields are given by
\bfl
\mathcal{K} := \{ (\phi_x, \phi_t) \in & C_0(\Omega\times\R; \R^d\times\R ):  \\ & \phi_t(x,t) + \rho(x,t) \geq \eta^*(\phi_x(x,t)), \\ & \forall (x,t) \in \Omega\times\R \} .
\label{eq:constraint_set_K}
\efl
If $v$ is indicator function of the subgraph of $u$, i.e. $v=1_u$, one has \cite{lifting_global_solutions,bouchitte2018duality}
\bfl
Dv = \nu_{\Gamma_u}(x,t) \Hd^d \llcorner \Gamma_u(x,t),
\label{eq:Dv}
\efl
where $\Gamma_u$ denotes the measure theoretic boundary of the subgraph (i.e. the complete graph of $u$ or the singular set of $1_u$) and $\nu_{\Gamma_u}$ the inner (downwards-pointing) unit normal to $\Gamma_u$. For smooth $u$, the latter is
\bfl
\nu_{\Gamma_u} = \frac{1}{\sqrt{1+ \| \nabla u \|^2}}\begin{pmatrix} \nabla u(x) \\ -1 \end{pmatrix}. \label{eq:nu}
\efl
See Fig.~\ref{fig:calibration} for a visualization.

In case of the $\TV$ regularizer \eqref{eq:tv_def}, $\eta$ is convex and one-homogeneous, and its Fenchel conjugate $\eta^*$ is the indicator function of a convex set. Therefore, the constraint in \eqref{eq:constraint_set_K} can be  separated \cite{sublabel_discretization}:
\begin{align}
\mathcal{F}(v) &= \sup_{\substack{\phi_x \in \mathcal{K}_x \\ \phi_t \in \mathcal{K}_t}} \ls \int_{\otg} \ll \begin{pmatrix} \phi_x, \phi_t \end{pmatrix}, Dv \rr \rs, \label{eq:lifted_Fxt} \\
\begin{split} \mathcal{K}_x := \{ & \phi_x \in  C_0(\Omega\times\R; \R^d):  \\ & \phi_x(x,t) \in \text{dom} \{ \eta^* \}, \quad \forall(x,t) \in  \Omega\times\R \}, \label{eq:constraint_set_Kx} \end{split} \\
\begin{split} \mathcal{K}_t := \{ & \phi_t \in  C_0(\Omega\times\R; \R): \\ - &\phi_t(x,t) \leq \rho(x,t),  \quad \forall (x,t) \in \Omega\times\R \}. \label{eq:constraint_set_Kt} \end{split}
\end{align}

In \cite{lifting_global_solutions,bouchitte2018duality}, the authors show that $F(u) = \mathcal{F}(1_u)$ holds for any $u\in W^{1,1}$. Moreover, if the non-convex set $ \{ 1_u : u\in W^{1,1}  \} $ is relaxed to the convex set
\bfl
C := \{ v \in & \BV_{\text{loc}} (\Omega\times\R, [0,1]): \\ & v(\cdot,t) = 1\; , \, \text{for a.e.~ } t \leq \min(\Gamma ), \\ & v(\cdot, t) = 0\;, \,  \text{for a.e.~ } t > \max(\Gamma)  \},\label{eq:setc}
\efl 
any minimizer of the lifted problem $\inf_{v\in C}\mathcal{F}(v)$ can be transformed into a global minimizer of the original nonconvex problem $\inf_{u\in W^{1,1}}\mathcal{F}(1_u)$  by  thresholding: \dbdetail The thresholding process does not change the energy, and produces a characteristic function of the form $1_u$ for some $u$ in the original function space, on which $\mathcal{F}(1_u)$ and $F(u)$ agree \cite[Thm.~3.1]{lifting_global_solutions}\cite[Thm.~4.1,Prop.~4.13]{bouchitte2018duality}. 

The lifting approach has also been connected to dynamical optimal transport and the Benamou-Brenier formulation, which allows to incorporate higher-order regularization \cite{vogt2020connection}. 

\paragraph{Lifted total variation.}
For $H\equiv 0$ and $J=\TV$ it turns out that $\phi_t \equiv 0$ is optimal in \eqref{eq:lifted_Fxt}-\eqref{eq:constraint_set_Kt}. This can either be derived from the fact that $(Dv)_t$ is non-positive (see \cite{lifting_global_solutions}) and that $\phi_t$ is non-negative due to the constraint \eqref{eq:constraint_set_Kt}. If $v = 1_u$ for sufficiently smooth $u$, one can also easily see $\phi_t\equiv 0$ by  applying \eqref{eq:Dv}-\eqref{eq:nu} and again arguing that $\phi_t\geq 0$ holds due to the constraint \eqref{eq:constraint_set_Kt}: \dbe
\begin{align*}
\mathcal{TV}(1_u) &= \sup_{\substack{\phi_x \in \mathcal{K}_x \\ \phi_t \in \mathcal{K}_t}} \ls \int_{\Gamma_u} \ll \begin{pmatrix} \phi_x \\ \phi_t \end{pmatrix}, \begin{pmatrix} \nabla u(x) \\ -1 \end{pmatrix} \rr \frac{\dH^d}{\sqrt{1+\|\nabla u\|^2}} \rs .
\end{align*}
Subsequently, we can reduce the lifted total variation for any $v=1_u$ (i.e., indicator functions of a subgraph of $u\in W^{1,1}$) to
\begin{align}
\mathcal{TV}(v) 
% &=  \sup_{\substack{\phi_x \in C_0(\otg; \, \R^d) \\ \| \phi_x(x,t) \|_2 \leq 1}} \ls \int_{\Gamma_u} \ll  \phi_x , \nabla u(x)  \rr \frac{\dH^d}{\sqrt{1+\|\nabla u\|^2}} \rs\\
&=  \sup_{\substack{\phi_x \in C_0(\otg; \, \R^d) \\ \| \phi_x(x,t) \|_2 \leq 1}} \ls \intog \ll \phi_x, (Dv)_x \rr \rs .
\end{align}
Furthermore,  the following equality holds \cite[Thm.~3.2]{lifting_global_solutions}: 
\bfl
\mathcal{TV}(1_u) = \TV(u).
\label{eq:tv_equality}
\efl

\subsection{Discrete Sublabel-Accurate Setting}
\label{sec:sa_relax}

\paragraph{Discrete setting.} 
In a fully discrete setting with discretized domain and finite range $\Gamma$,
Ishikawa and Geiger proposed  first lifting strategies for the labeling problem 
\cite{ishikawa_geiger,ishikawa}. Later the relaxation of the labeling problem was studied in a spatially continuous setting with binary \cite{chan_vese,chan_relax} and multiple labels \cite{depth2,lifting_continuous_multiclass}. 

\paragraph{Sublabel-accurate discretization.}
In practice, a straightforward discretization of the label space $\Gamma$ during the implementation process leads to artifacts and the quality of the solution strongly depends on the number and positioning of the chosen discrete labels. Therefore, it is advisable to employ a \emph{sublabel-accurate} discretization \cite{sublabel_cvpr}, which allows to preserve information about the data term in between discretization points, resulting in smaller problems.
In \cite{sublabel_discretization}, the authors point out that this approach is closely linked to the approach in \cite{lifting_global_solutions} when a combination of piecewise linear and piecewise constant basis functions is used for discretization. We also refer to \cite{mollenhoff2019lifting} for  an extension of the sublabel-accurate lifting approach to arbitrary convex regularizers.

For reference, we provide a short summary of the lifting approach with sublabel-accurate discretization for TV-\eqref{eq:problem} problems using the notation from \cite{sublabel_cvpr}. The approach comprises three steps:

\paragraph{Lifting of the label space.} First, we choose $L$ labels $ \gamma_1 < \gamma_2 < ... < \gamma_L $ such that $ \Gamma = [ \gamma_1, \gamma_L ] $. These labels decompose the label space $\Gamma$ into $l := L-1$ \emph{sublabel spaces} $\Gamma_i := [\gamma_i, \gamma_{i+1}]$. Any value in $\Gamma$ can be written as
\begin{equation}
\gamma_i^{\alpha} := \gamma_i + \alpha (\gamma_{i+1} - \gamma_i ),
\label{eqdef:sa_gammaia}
\end{equation}
for some $i \in \{1, 2, ..., l\}$ and $\alpha \in [0,1]$. The lifted representation of such a value in $\R^l$ is defined as
\begin{equation}
\bs{1}_i^{\alpha} := \alpha \bs{1}_i + (1-\alpha )\bs{1}_{i-1} \in \R^l,
\end{equation}
where $\bs{1}_i \in \R^l$ is the vector of $i$ ones followed by $l-i$ zeroes. The \emph{lifted label space} -- which is non-convex -- is given as
$\boldsymbol{\Gamma} := \{ \bs{1}_i^{\alpha}  \in\R^l |  i \in \{ 1,2,...,l \},\, \alpha \in [0,1] \}$.

 If $\bs{u}(x) = \bs{1}_i^{\alpha} \in \boldsymbol{\Gamma}$ for (almost) every $x$, it can be  mapped uniquely to the equivalent value in the unlifted (original) label space by  
\begin{equation}
u(x) = \gamma_1 + \sum_{i=1}^l \bs{u}^i(x)(\gamma_{i+1} - \gamma_i ).
\label{eq:projection}\end{equation}
We refer to such functions $\bs{u}$ as \emph{sublabel-integral.} 

\paragraph{Lifting of the data term.} Next, a lifted formulation of the data term is derived that in effect approximates the energy locally convex between neighbouring labels,  justifying  the ``sublabel-accurate'' term.
For the possibly non-convex data term of \eqref{eq:problem}, the lifted -- yet still non-convex -- representation 
is defined as $\bs{\rho}(x,\cdot):\R^l \mapsto \Rc$, 
\begin{align}
\bs{\rho}(x,\bs{u}(x)) := \inf_{\substack{i\in\{1,...,l\}, \\ \alpha \in [0,1]}}  \left \{ \rho(x,\gamma_i^{\alpha}) + \delta_{\bs{1}_i^{\alpha}}(\bs{u}(x)) \right \}.
\label{eq:def_data_integrand}
\end{align} 
Note that the domain is $\R^l$ and not just $\bs{\Gamma}$. Outside of the lifted label space $\boldsymbol{\Gamma}$  the lifted representation $\bs{\rho}$ is set to $\infty$. Applying the definition of Legendre-Fenchel conjugates twice  with respect to the second variable  results in a relaxed -- and convex -- data term:
\begin{equation}
 \bs{H}(\bs{u}) = \int_{\Omega} \bs{\rho}^{**}(x, \bs{u}(x)) dx.
\label{eq:def_data}
\end{equation}
For explicit expressions of $\bs{\rho}^{**}$ in the linear and non-linear case we refer  to \cite[Prop.~1, Prop.~2]{sublabel_cvpr}.

\input{Sections/figure_sa.tex}
\paragraph{Lifting of the total variation regularizer.} Lastly, a lifted representation of the (isotropic) total variation regularizer is established, building on the theory developed in the context of multiclass labeling approaches \cite{lifting_continuous_multiclass,lifting_tv_local_envelope}.
The method heavily builds on representing the total variation regularizer with the help of \emph{Radon measures} $D\bs{u}$. For further details we refer the reader to \cite{book_ambrosio}.
The lifted -- and non-convex -- integrand $\boldsymbol{\phi}: \R^{l\times d} \mapsto \Rc$ is defined as
\begin{align}
\boldsymbol{\phi}(\bs{g}) := &\inf_{\substack{ 1\leq i\leq j\leq l, \\ \alpha, \beta \in [0,1]}} \ls
| \gamma_i^{\alpha} - \gamma_j^{\beta} | \cdot \| v \|_2 + \delta_{(\bs{1}_i^{\alpha} - \bs{1}_j^{\beta}) v^{\top}}(\bs{g}) \rs. 
\label{eq:def_reg_integrand}
\end{align}
Applying the definition of Legendre-Fenchel conjugates twice  results in a relaxed -- and convex -- regularization term:
\begin{flalign}
\bs{TV}(\bs{u}) := \int_{\Omega} \boldsymbol{\phi}^{**}(D\bs{u}),
\label{eq:def_reg}
\end{flalign}
where $D\bs{u}$ is the distributional derivative in the form of a Radon measure. For isotropic TV, it can be shown that, for  $\bs{g}\in\R^{l\times d}$, 
\begin{flalign}
&\bs{\phi}^{**}(\bs{g}) = \sup_{\bs{q}\in\bs{K}_{\text{iso}}} \ls \langle \bs{q}, \bs{g} \rangle \rs , \\
&\bs{K}_{\text{iso}} = \left\{ \bs{q}\in\R^{l\times d}  \middle| \|  \bs{q}_i \|_2 \leq \gamma_{i+1} -\gamma_i,  \forall i=1,...,l \right\}.
\label{eq:K_iso}
\end{flalign}  
For more details we refer to \cite[Prop.~4]{sublabel_cvpr} and \cite{lifting_tv_local_envelope}. 

Unfortunately isotropic TV in general does not allow to prove global optimality for the discretized system,  as there is no known coarea-type formula for the discretized isotropic case.  Therefore, we also consider the lifted \emph{anisotropic} ($L^1$) TV,  i.e., $\TV_{\text{an}}(u):=\int_\Omega \|Du\|_1$. With the same strategy as in the isotropic case \eqref{eq:K_iso}, one obtains 
\begin{flalign}
\label{eq:K_an1}
\bs{K}_{\text{an}} = \left\{ \bs{q}\in\R^{l\times d}  \middle|  \| \bs{q}_i \|_\infty \leq \gamma_{i+1} -\gamma_i,  \forall i = 1,...,l \right\}\\
= \bigcap_{j=1, \hdots  ,d} \left\{  \bs{q}  \middle|   \| \bs{q}_{i,j} \|_2 \leq \gamma_{i+1} - \gamma_i ,  \forall i=1,...,l \right\} .
\label{eq:K_an2}
\end{flalign} 
\begin{proof}[Derivation of $\bs{K}_{\text{an}}$]
Here the maximum norm in {\eqref{eq:K_an1}} originates as the dual norm to $\|\cdot\|_1$. To see the equality to {\eqref{eq:K_an2}}, let $\bs{q} \in \R^{l\times d}$. Then
\begin{flalign}
 \| \bs{q}_i \|_\infty &\leq \gamma_{i+1} - \gamma_i , \, \forall i=1,\hdots , l ,   \\
\Leftrightarrow \max_{j=1, \hdots  ,d} | \bs{q}_i^j | & \leq \gamma_{i+1} - \gamma_i  , \,  \forall  i=1,\hdots , l.   
\end{flalign}
Since $q_i^j$ is a scalar, this is equivalent to
\begin{flalign}
\max_{j=1, \hdots  ,d} \| \bs{q}_{i,j} \|_2 \leq \gamma_{i+1} - \gamma_i  ,  \, &\forall  i=1,\hdots , l , \\
\begin{split}
 \Leftrightarrow  \| \bs{q}_{i,j} \|_2  \leq \gamma_{i+1} - \gamma_i  ,  \, &\forall  i = 1, ..., l, \\ \, &\forall j=1,\hdots , d.
\end{split}
\end{flalign}
This shows the equality \eqref{eq:K_an1}--\eqref{eq:K_an2}. \qed 
\end{proof}

Together, the previous three sections allow us to formulate a version of the problem of minimizing the lifted energy \eqref{eq:lifted_F} over the relaxed set \eqref{eq:setc} that is discretized in the label space $\Gamma$:
\begin{equation}
\inf_{\bs{u}\in\BV(\Omega,\bs{\Gamma})} \ls \int_{\Omega} \bs{\rho}^{**}(x, \bs{u}(x)) d x +  \int_{\Omega} \boldsymbol{\phi}^{**}(D\bs{u}) \rs.
\end{equation}
Once the non-convex set $\bs{\Gamma}$ is relaxed to its convex hull, we obtain a fully convex lifting of problem TV-\eqref{eq:problem} similar to \eqref{eq:prob-lf}, which can now be spatially discretized.

%% file: Sections/lifted_breg.tex
\section{Equivalence of the Lifted Bregman Iteration in the Continuous Setting}
\label{sec:lifted_bregman}

An interesting question is to find conditions under which the Bregman iteration (Alg.~1) and the more general lifted Bregman iteration (Alg.~2) are equivalent in the following sense:
If $u_k$ is a solution of the original Bregman iteration, then its lifted representation (i.e., the indicator function of its subgraph $1_{u_k}$ in the fully continuous setting) is a solution $v_k$ of the lifted Bregman iteration; or \emph{the} solution if it is unique. The exact definition is not trivial, as there is potential ambiguity in choosing the subgradient term in the lifted setting.

In this section, we consider the problem in the function space with continuous range $\Gamma$, before moving on to the discretized setting in the later sections.

\subsection{Subdifferential of the Total Variation}
\label{sss:subdiff_tv}

The Bregman iteration crucially requires elements from the subdifferential $\partial J$ of the regularizer. Unfortunately, for the choice $J=\TV$, which is the basis of the classical inverse scale space flow, this requires to study elements from $\BV^\ast$, i.e., the dual space of $\BV$, which is not yet fully understood. 

In order to still allow a reasonably accurate discussion, we make two simplifying assumptions: Firstly, we restrict ourselves to the case
$u\in W^{1,1}$ and $\Omega \subseteq \R^2$, which allows to embed $W^{1,1} \hookrightarrow L^2$ by the Sobolev embedding theorem \cite[Thm.~10.9]{book_alt} 

Secondly, we will later assume that the subgradients can be represented as $L^2$ functions. While this is a rather harsh condition, and there are many subgradients of $\TV$ outside of this restricted class, these assumptions still allow to formulate the major arguments in an intuitive way without being encumbered by too many technicalities.

The total variation \eqref{eq:tv_def} can be viewed as a support function:
\begin{align}
&\TV(u) = \sigma_\Psi (u) \\
&\Psi := \ls - \mydiv \psi \middle| \psi \in C_c^\infty(\Omega; \, \R^d) ,  \|\psi \|_\infty \leq 1 \rs .
 \label{eq:Psi_TV}
\end{align}
Using the relation $\sigma_C^\ast = \delta_{\text{cl } C}$ for convex sets $C$, its Fenchel conjugate is \cite{book_ekeland}[Def.~I.4.1, Example I.4.3] 
\begin{align}
&\TV^*(p) = \sigma^*_\Psi (p) = \delta_{\mycl \Psi} (p). 
\label{eq:fenchel_tv}
\end{align}
According to \cite[Prop.~I.5.1]{book_ekeland} it holds $p\in\partial\TV(u)$ iff 
\begin{align}
\TV(u) =  \langle u, p \rangle - \delta_{\mycl \Psi}(p).
\end{align}
The closure of $\Psi$ with respect to the $L^2$ norm (note that we have restricted ourselves to $L^2$) is \cite[proof of Prop.~7]{bredies2016pointwise}  
\bfl
\mycl \Psi = \ls - \mydiv \psi \, | \, \psi \in W_0^2 (\mydiv; \Omega) , \, \| \psi \|_\infty \leq 1\rs ,
\label{eq:clpsi}
\efl
where
\begin{align}
W^2 (\mydiv; \Omega) &:= \ls \psi \in L^2(\Omega, \R^d) \, | \,  \mydiv \psi \in L^2(\Omega) \rs , \\
\| \psi \|_{W^2 (\mydiv)} &:= \| \psi \|_{L^2}^2 + \| \mydiv \psi \|_{L^2}^2, \\
W_0^2 (\mydiv; \Omega) &:= \overline{C_c^\infty(\Omega; \R^d)}^{\| \cdot \|_{W^2 (\mydiv)}}.
\end{align}
Consequently, in our setting with $u\in W^{1,1}\subset L^2$, we know that $u^\ast\in L^2$ is a subgradient iff $u^\ast = -\mydiv \psi$ for some $\psi\in W_0^2(\mydiv;\Omega)$ which satisfies $\|\psi\|_\infty \leq 1$. Furthermore, it holds $TV(u)=\langle u, -\mydiv \psi\rangle_{L^2}$ \cite[Prop.~7]{bredies2016pointwise}.

\subsection{Lifted Bregman Iteration (Continuous Setting)}
\label{sss:lifted_bregman}

The lifted Bregman iteration (Alg.~2) is conceptualized in the lifted setting, i.e., applying the Bregman iteration on a lifted problem and choosing a subgradient of the lifted regularizer. Here, we will perform the lifting on the original Bregman iteration, i.e., lift \eqref{eq:breg_uk} for a given subgradient $p_{k-1}\in\partial\TV(u_{k-1})$. 

We assume that $u_{k-1}\in W^{1,1}$ and $p_{k-1}\in L^2(\Omega)$ are given, such that  $p_{k-1}\in\partial\TV(u_{k-1})$ holds. We regard the linear Bregman term as part of the data term. Using the theory of calibration-based lifting \eqref{eq:lifted_Fxt}-\eqref{eq:constraint_set_Kt}, the lifted version of \eqref{eq:breg_uk} is:
\begin{align}\label{eq:FvK}
\mathcal{F}_{\text{Breg}}(v) &= \sup_{\substack{\phi_x \in \mathcal{K}_x \\ \phi_t \in \mathcal{K}_t}} \ls \int_{\Omega \times \Gamma} \left\langle \begin{pmatrix}
\phi_x \\ \phi_t
\end{pmatrix}, Dv \right\rangle \rs , \\
\begin{split}
\label{eq:K_x}
\mathcal{K}_x &:= \{ \phi_x \in  C_0(\Omega\times\R; \R^d):  \\ &\qquad  \phi_x(x,t) \in \text{dom} \{ \eta^* \}, \quad \forall (x,t) \in \Omega\times\R \}, \end{split} \\
\begin{split} \mathcal{K}_t &:= \{ \phi_t \in  C_0(\Omega\times\R; \R ):  \\ &\qquad  - \phi_t(x,t) + tp_{k-1}(x) \leq \rho(x,t), \\ & \qquad \text{a.e. } x\in\Omega, \forall t \in \R \} \label{eq:K_t}\end{split}
\end{align}

The term $t p_{k-1}(x)$ comes from the integrand in the Bregman term,
\begin{equation}
\ll p_{k-1}, u \rr = \int_\Omega u(x) p_{k-1}(x) \dx.
\end{equation}

\subsection{Sufficient Condition for Equivalence (Continuous Case)}
\label{sss:sufficient_continuous}

The following proposition shows that the Bregman iteration (Alg.~1) and the fully continuous formulation of the lifted Bregman iteration (Alg.~2) are equivalent as long as the subgradients used in either setting fulfil a certain condition. We have to assume unique solvability of the original problem, as is the case for strictly convex functionals such as ROF. 

\begin{proposition} Assume that the minimization prob-\linebreak lems~\eqref{eq:breg_uk} in the original Bregman iteration have unique solutions $u_k$. Moreover, assume that the solutions $v_k$ in the lifted iteration \eqref{eq:breg_uk_lifted} are \emph{integral,} i.e., indicator functions of subgraphs $v_k = 1_u$ of some $u\in W^{1,1}$. If the chosen subgradients $p_{k-1}\in\partial \TV (u_{k-1})$ and $\tilde{p}_{k-1} \in \partial \mathcal{TV} (v_{k-1})$ satisfy 
\bfl 
\tilde{p}_{k-1}(\cdot,t) = p_{k-1}(\cdot), \quad \text{for a.e. } t \in \Gamma ,
\label{eq:lifted_grad_cont}
\efl
then the iterates $v_k$ of the lifted Bregman iteration are the indicator functions of the subgraphs of the iterates $u_k$ of the original Bregman iteration, i.e., $v_k = 1_{u_k}$. \label{prop:equivalency_cont}
\end{proposition}
\begin{proof}
We first show that the Bregman iteration for the lifted energy with this specific choice of $\tilde{p}_{k-1}$ is simply the lifted Bregman energy for the subgradient $p_{k-1}$ (note that this is not necessarily the case for an arbitrary subgradient of the lifted energy).

In order to do so, we substitute $\tilde{\phi_t}(x,t) := \phi_t(x,t) - t p_{k-1}(x)$ in \eqref{eq:FvK}--\eqref{eq:K_t} and rewrite the problem as 
\begin{align}
\label{eq:Fn_Breg}
\begin{split} \F_{\text{Breg}}(v) &:= \sup_{\substack{\phi_x \in \mathcal{K}_x \\ \tilde{\phi_t} \in \tilde{\mathcal{K}}_t}} \ls \intog \ll \begin{pmatrix} \phi_x \\ \tilde{\phi}_t \end{pmatrix}, Dv \rr \right. \\ 
& \qquad \qquad \qquad \left. - \int_\Gamma \ll p_{k-1}, v(\cdot,t) \rr \dt \rs , \end{split} \\ 
\begin{split} \tilde{\mathcal{K}}_t &:= \{ \tilde{\phi}_t \in  C_0(\Omega\times\R; \R ):  \\ &\qquad  - \tilde{\phi}_t(x,t)  \leq \rho(x,t), \forall (x,t) \in \Omega\times\R \} \end{split}\end{align}
with $\mathcal{K}_x$ from \eqref{eq:K_x}. 
If $\tilde{p}_{k-1}(\cdot,t) = p_{k-1}(\cdot)$ as in the assumption, we see that the second term in \eqref{eq:Fn_Breg} is simply $-\ll \tilde{p}_{k-1}, v \rr$, i.e., $\mathcal{F}_{\text{Breg}}(v) = \mathcal{F}(v)-\ll \tilde{p}_{k-1}, v\rr$: Adding the linear Bregman term with this specific $\tilde{p}_{k-1}$ results in the same energy as lifting the original Bregman energy including the $\ll p_{k-1}, u \rr$ term.

Therefore, for any integral solution $v=\bs{1}_{u}$, the function $u$ must be a solution of the original Bregman energy. Due to the uniqueness, this means $u=u_k$. \qed
\end{proof}

\subsection{Existence of Subgradients Fulfilling the Sufficient Condition}
\label{sss:subgradients_fulilling}
One question that remains is whether subgradients $\tilde{p}_{k-1}$ as required in Prop.~\ref{prop:equivalency_cont} actually exist. 
In this section, we show that this is the case for $J=\TV$.

For fixed $p\in\partial \TV(u)\subseteq BV^\ast(\Omega)$ we define $\tilde{p}\in BV^\ast(\Omega\times\Gamma)$ by $\langle \tilde{p}, v \rangle := \int_\Gamma \ll p, v(\cdot,t) \rr dt$. If $p$ is a function, this corresponds to setting $\tilde{p}$ constant copies of $p$ along the $\Gamma$ axis, i.e., $\tilde{p}(x,t):=p(x) \, \forall t\in\Gamma$.
Similar to the previous paragraphs, if $u$ is a $W^{1,1}$ and therefore (in our setting) $L^2$ function, and using our general simplifying assumption that $p\in L^2(\Omega)$, we know that $\tilde{p}\in L^2(\Omega\times\Gamma)$ as defined due to the boundedness of $\Gamma$.

Therefore, similar to \cite[Prop.~7]{bredies2016pointwise} and \cite[Example I.4.3, Prop.~I.5.1]{book_ekeland}, $\tilde{p}$ is a subgradient of $\mathcal{TV}$ at $\bs{1}_{u}$ iff
\bfl
\mathcal{TV}(1_{u}) = \ll \tilde{p}, 1_{u} \rr - \mathcal{TV}^*(\tilde{p}).
\label{eq:to_show}
\efl
From section \ref{ss:contset}, we recall 
\begin{align}\label{eq:psixdef}
&\mathcal{TV}(v) = \sup_{z\in\Phi_x} \ls \intog \ll z, v \rr \dxt \rs = \sigma_{\Phi_x} (v)  \\
 &\Phi_x := \{ -\mydivx \phi_x |  \phi_x \in C_0(\otg; \, \R^d) ,   \| \phi_x \|_\infty \leq 1 \} 
\end{align}
and, therefore,
\begin{align}
\mathcal{TV}^*(\tilde{p}) &= \delta_{\mycl \Phi_x} (\tilde{p}),
\end{align}
where the closure is taken with respect to the $L^2(\Omega\times\Gamma)$ norm. 
Therefore, if we can show that $\tilde{p}\in \mycl \Phi_x$ and $\mathcal{TV}(1_{u}) = \ll \tilde{p}, 1_{u} \rr$, by \eqref{eq:to_show} we know that $\tilde{p}\in\partial \mathcal{TV}(\bs{1}_{u})$.

The fact that $\tilde{p} \in \mycl \Phi_x$ follows directly from $p \in \mycl \Psi$ with $\Psi$ as in \eqref{eq:Psi_TV}: For every sequence $\Psi\supset\left(p_n\right) \to p$ we have a sequence $\psi_n$ of $C_c^\infty(\Omega)$ functions with $\|\psi(x)\|_2\leq 1$ and $p_n=-\mydiv \psi_n$. Thus we can set $(\phi_x)_n(x,t):=\psi_n(x)$, so that $-\mydiv_x (\phi_x)_n\in\Phi_x$ from \eqref{eq:psixdef} and $-\mydiv_x (\phi_x)_n (x,t) = -\mydiv \psi_n(x)=p_n(x)$. Thus $-\mydiv_x (\phi_x)_n(\cdot,t)\to p$ in $L^2(\Omega)$ for all $t\in\Gamma$. Due to the boundedness of $\Gamma$, this implies $-\mydiv_x (\phi_x)_n \to \tilde{p}$ in $L^2(\Omega\times\Gamma)$, which shows $\tilde{p}\in\mycl \Phi_x$ as desired. 

In order to show the final missing piece in \eqref{eq:to_show}, i.e., $\mathcal{TV}(1_{u}) = \ll \tilde{p}, 1_{u} \rr$, note that
% \p  can define $\tilde{p}_n(x,t):=p_n(x)$ for all $t\in\Gamma$. Due to the boundedness of $\Gamma$, this implies $$\tilde{p}_n \to set $_{n\in\N}$ with ß
% \bfl
%         \lim_{n\to\infty} \| p_n \|_{W^2(\mydiv)} = \| p \|_{W^2(\mydiv)} 
% \efl
% we define $\left(\tilde{p}_n\right)_{n\in\N}$ as $\tilde{p}_n (\cdot, t) := p_n(\cdot), \, \forall t \in \Gamma$. Since $\Gamma$ is bounded and all $p_n$ are constant with respect to the second variable for all $t \in \Gamma$, it holds $\tilde{p}_n \in C_0(\otg; \R^d)$ and
% \begin{align}
%         &\lim_{n\to\infty} \| \tilde{p}_n \|_{W^2(\mydivx)} \\
%         &=  \lim_{n\to\infty} \into \int_\Gamma \left( \tilde{p}_n(x,t) \right)^2 + \left( \mydivx \tilde{p}_n(x,t) \right)^2 \dt \dx \\
%         &=  \lim_{n\to\infty} \into | \Gamma | \left( \left( p_n(x) \right)^2 + \left( \mydiv p_n(x) \right)^2 \right)  \dx \\
%         &=  \into | \Gamma | \left( \left( p(x) \right)^2 + \left( \mydiv p(x) \right)^2 \right) \dx \\
%         &=  \intog \left( \tilde{p}(x,t) \right)^2 + \left( \mydivx \tilde{p}(x,t) \right)^2 \dt \dx \\
%         &= \| \tilde{p} \|_{W^2(\mydivx)}.
% \end{align}
%With \eqref{eq:tv_equality} and the coarea formula the rest of \eqref{eq:to_show} follows:
\begin{align}
\mathcal{TV}(1_{u}) &\overset{\eqref{eq:tv_equality}}{=} \TV(u) = \ll p, u \rr\\
 & \overset{(*)}= \int_\Omega p(x) \int_\Gamma 1_{u}(x,t) \dt \dx\\
 & = \langle \tilde{p}, 1_{u}\rangle.
\end{align}
The crucial step is $(*)$, where we again used a coarea-type formula for linear terms.

Therefore, by defining $\tilde{p}$ based on $p$ as above, we have recovered a subgradient of the lifted regularizer $\mathcal{TV}$ of the form required by Prop.~\ref{prop:equivalency_cont}.

%% file: Sections/lifted_breg_half.tex
\section{Equivalence in the Half-Discretized Formulation}

In the previous section, we argued in the function space, i.e., $u\in W^{1,1}(\Omega\times\Gamma)$. While theoretically interesting, this leaves the question whether a similar equivalency between the original and lifted Bregman iterations can also be formulated after discretizing the range;  i.e., in the sublabel-accurate lifted case. 

For simplicity, the following considerations are formal due to the mostly pointwise arguments. However, they can equally be understood in the spatially discrete setting with finite $\Omega$, where arguments are more straightforward. For readability, we consider a fixed $x \in \Omega $ and omit $x$ in the arguments.

\subsection{Lifted Bregman Iteration (Discretized Case)}

Analogous to our argumentation in the continuous setting, we first perform the (sublabel-accurate) lifting on the equation from the original Bregman iteration \eqref{eq:breg_uk}, assuming that a subgradient $p_{k-1}\in\partial\TV(u_{k-1})$ is given. We show, that the extended data term $H_\text{Breg}(u) := H(u) - \ll p_{k-1}, u \rr$ has a lifted representation of the form $\bs{H}_\text{Breg}(\bs{u}) := \bs{H}(\bs{u}) - \ll p_{k-1}\bs{\tilde{\gamma}}, \bs{u} \rr$, similar to \eqref{eq:breg_uk_lifted}. Again, it is not clear whether $p_{k-1}\bs{\tilde{\gamma}}$ is a subgradient of the lifted total variation. In the following, in a slight abuse of notation,  we use pointwise arguments for fixed $x\in\Omega$, e.g. $u=u(x),$ $p=p(x)$, etc.

\begin{proposition}
Assume $\rho_1,\rho_2,h: \Gamma \mapsto \Rc$ with
\begin{align}
\rho_2(u) := \rho_1 (u) - h(u), \quad h(u) := pu, \quad p\in\R,
\end{align}
where $\rho_1$ and $\rho_2$ should be understood as two different  data terms in~\eqref{eq:problem}. Define
\begin{equation}
\tilde{\bs{\gamma}} := \begin{pmatrix} \gamma_2 - \gamma_1, & \hdots, & \gamma_{L} - \gamma_l \end{pmatrix}^{\top}\label{eq:bs_gamma}
\end{equation}
Then, for the lifted representations $\bs{\rho}_1,\bs{\rho}_2,\bs{h} : \R^l \mapsto \Rc$ in \eqref{eq:def_data_integrand}, it holds 
\begin{equation}
\bs{\rho}_2^{**}(\bs{u})  =  \bs{\rho}_1^{**}(\bs{u}) - \bs{h}^{**}(\bs{u})   =   \bs{\rho}_1^{**}(\bs{u}) - \langle p \tilde{\bs{\gamma}}, \bs{u} \rangle  .
\end{equation}
\label{prop:lifted_sum}
\end{proposition}

\begin{proof}
We deduce the biconjugate of $\bs{\rho}_2$ step-by-step and show that the final expression implies the anticipated equality. According to \eqref{eq:def_data_integrand} the lifted representation of $\rho_2$ is
\begin{align}
\bs{\rho}_2(\textbf{u}) &= \inf_{\substack{i\in\{1,...,l\}, \\\alpha \in [0,1]}}  \left \{ \rho_2(\gamma_i^{\alpha}) + \delta_{\textbf{1}_i^{\alpha}}(\textbf{u}) \right \}.
\label{eq:def_data_integrand_sum}
\end{align}
We use the definition of the Fenchel conjugate and note that the supremum is attained for some  $\textbf{u}\in\boldsymbol{\Gamma}$:
\begin{align}
\bs{\rho}_2^*(\textbf{v}) &= \sup_{\textbf{u} \in \R^l} \left \{ \langle \textbf{u}, \textbf{v} \rangle - \bs{\rho}_2(\textbf{u}) \right \} \\
&= \sup_{\substack{j\in\{1,...,l\}, \\ \beta \in [0,1]}} \left \{ \langle \textbf{u}, \textbf{v} \rangle - \bs{\rho}_2(\textbf{u}) \right \}.
\end{align}
The definition of $\bs{\rho}_2$ and $\rho_2$ lead to:
\begin{align}
\bs{\rho}_2^*(\textbf{v}) 
&= \sup_{\substack{j\in\{1,...,l\}, \\ \beta \in [0,1]}} \left \{ \langle \textbf{1}_j^{\beta}, \textbf{v} \rangle -  \rho_1\left(\gamma_j^{\beta}\right) + p \gamma_j^{\beta} \right \} .
\end{align}
Using $\tilde{\bs{\gamma}}$ as in \eqref{eq:bs_gamma} we can furthermore express $p\gamma_j^{\beta} $ in terms of $\textbf{1}_j^{\beta}$:
\begin{align}
\bs{\rho}_2^*(\textbf{v}) &= \sup_{\substack{j\in\{1,...,l\},\\ \beta \in [0,1]}} \left \{ \langle \textbf{1}_j^{\beta}, \textbf{v} \rangle - \rho_1\left(\gamma_j^{\beta}\right) + \langle p \tilde{\bs{\gamma}} , \textbf{1}_j^{\beta} \rangle  \right \} \\
&= \sup_{\substack{j\in\{1,...,l\}, \\ \beta \in [0,1]}} \left \{ \langle \textbf{1}_j^{\beta}, \textbf{v} + p \tilde{\bs{\gamma}} \rangle - \rho_1\left(\gamma_j^{\beta}\right) \right \}.
\end{align}
Next we compute the biconjugate of $\bs{\rho}$:  
\begin{align}
\bs{\rho}_2^{**}(\textbf{w}) &= \sup_{\textbf{v} \in \R^l} \left \{\langle \textbf{v}, \textbf{w} \rangle - \bs{\rho}_2^*(\textbf{v}) \right \} .
\end{align}
By substituting $\textbf{z} := \textbf{v} + p \tilde{\bs{\gamma}} $ we get
\begin{align}
\bs{\rho}_2^{**}(\textbf{w}) =  &
\sup_{\textbf{z} \in \R^l} \Bigg \{ \langle \textbf{z} -p \tilde{\bs{\gamma}}, \textbf{w} \rangle  \\ &   -   \sup_{\substack{j\in\{1,...,l\}, \\ \beta \in [0,1]}} \left \{ \langle \textbf{1}_j^{\beta}, \textbf{z}  \rangle - \rho_1\left(\gamma_j^{\beta}\right) \right \} \Bigg \} \\
= & - \langle \textbf{w}, p \tilde{\bs{\gamma}} \rangle + \sup_{\textbf{z} \in \R^l} \Bigg \{ \langle \textbf{z} , \textbf{w} \rangle  \\ &  - \sup_{\substack{j\in\{1,...,l\}, \\ \beta \in [0,1]}} \left \{ \langle \textbf{1}_j^{\beta}, \textbf{z}  \rangle - \rho_1\left(\gamma_j^{\beta}\right) \right \}  \Bigg  \}  \\
= & \bs{\rho}_1^{**}(\textbf{w}) - \langle \textbf{w}, p \tilde{\bs{\gamma}} \rangle .
\end{align}
In reference to \cite[Prop.~2]{sublabel_cvpr} we see, that the expression $ \langle \textbf{w}, p \tilde{\bs{\gamma}} \rangle $ is in fact $\textbf{h}^{**}(\textbf{w})$. This concludes the proof of Thm.~\ref{prop:lifted_sum}. \qed
\end{proof}

\subsection{Sufficient Condition for Equivalence (Discretized Setting)}

In the previous section, we performed the lifting on~\eqref{eq:breg_uk} of the original Bregman iteration for a fixed $p_{k-1}$. In this section, we show that -- under a sufficient condition on the chosen subgradients -- we can equivalently perform the Bregman iteration on the lifted problem where a subgradient $\bs{p}_{k-1}$ is chosen in the lifted setting (Alg.~2). This is the semi-discretized version of Prop.~{\ref{prop:equivalency_cont}}: 

\begin{proposition} Assume that the minimization problems \eqref{eq:breg_uk} in the original Bregman iteration have unique solutions. Moreover, assume that in the lifted iteration, the solutions $\bs{u}_k$ of \eqref{eq:breg_uk_lifted} in each step satisfy $\bs{u}_k(x)\in\bs{\Gamma}$, i.e., are sublabel-integral. If at every point $x$ the chosen subgradients $p_{k-1}\in\partial J(u_{k-1})$ and $\tilde{\bs{p}}_{k-1} \in\partial \bs{J}(\bs{u}_{k-1})$ satisfy 
\bfl 
\tilde{\bs{p}}_{k-1}(x) \quad = 
\quad p_{k-1}(x) \tilde{\bs{\gamma}} 
\label{eq:lifted_grad}
\efl
with $\tilde{\bs{\gamma}}$ as in \eqref{eq:bs_gamma}, then the lifted iterates $\bs{u}_k$ correspond to the iterates $u_k$ of the classical Bregman iteration \eqref{eq:breg_uk} according to \eqref{eq:projection}.
\label{prop:equivalency}
\end{proposition}

\begin{proof}[Proof of Proposition \ref{prop:equivalency}]
We define the \emph{extended data term}
\begin{equation}
H_{\text{Breg}}(u) := \into \rho(x,u(x)) - p(x) u(x) \dx,
\end{equation}
which incorporates the linear term of the Bregman iteration.
Using Prop.~\ref{prop:lifted_sum}, we reach the following lifted representation:
\begin{align}
\bs{H}_{\text{Breg}}(\bs{u}) &= \into \bs{\rho}^{**}(x,\bs{u}(x)) - \langle p(x)\tilde{\bs{\gamma}}, \bs{u}(x) \rangle \dx .
\end{align}
Hence the lifted version of \eqref{eq:breg_uk} is
\begin{equation}
\arg\min_{\bs{u}\in\bs{U}} \left \{ \bs{H}(\bs{u}) +\bs{J}(\bs{u}) - \langle p_{k-1} \tilde{\bs{\gamma}}, \bs{u} \rangle  \right \}.
\end{equation}
Comparing this to  \eqref{eq:breg_uk_lifted} shows that the minimization problem in the lifted iteration is the lifted version of \eqref{eq:breg_uk} if the subgradients  $p_{k-1}\in\partial J(u_{k-1})$ and $\tilde{\bs{p}}_{k-1} \in\partial \bs{J}(\bs{u}_{k-1})$ satisfy $ \tilde{\bs{p}}_{k-1} = p_{k-1} \tilde{\bs{\gamma}}$.
 In this case, since we have assumed that the solution of the lifted problem \eqref{eq:breg_uk_lifted} is sublabel-integral, it can be associated via \eqref{eq:projection} with the solution of the original problem \eqref{eq:breg_uk}, which is unique by assumption.\qed
\end{proof}
Thus, under the condition of the proposition, the lifted and unlifted Bregman iterations are equivalent.

%% file: Sections/numerical_discussion.tex
\input{Sections/figure_rof.tex}

\section{Fully-Discretized Setting}
\label{sec:numerical_discussion}

In this section, we consider the spatially discretized problem on a finite discretized domain~$\Omega^h$ with grid spacing $h$. 
In particular, we will see that the subgradient condition in Prop.~\ref{prop:equivalency} can be met in case of anisotropic TV and how such subgradients can be obtained in practice.

\subsection{Finding a Subgradient}

The discretized, sublabel-accurate relaxed total variation is  of the form
\begin{flalign}
        &\bs{J}^h(\nabla \bs{u}^h) = \max_{\bs{q}^h:\Omega^h \rightarrow\R^{k\times d}} Z, \\
        &Z := \left \{ \sum_{x \in \Omega^h}  \langle \bs{q}^h(x), \nabla \bs{u}^h(x) \rangle - \delta_\mathcal{K}(\bs{q}^h(x)) \right \},
\label{eq:q}
\end{flalign}
with $\bs{K}$ defined by \eqref{eq:K_iso} or \eqref{eq:K_an1}-\eqref{eq:K_an2} and $\nabla$ denoting the discretized forward-difference operator. By standard convex analysis (\cite[Thm.~23.9]{book_rock_convex}, \cite[Cor.~10.9]{book_rock_variational}, \cite[Prop.~11.3]{book_rock_variational}) we can show that if $\bs{q}^h$ is a  maximizer of \eqref{eq:q}, then $\bs{p}^h := \nabla^{\top}\bs{q}^h$ is a subgradient of $J^h(\nabla \bs{u}^h)$. Thus, the step of choosing a subgradient \eqref{eq:breg_qk_lifted} boils down to $\bs{p}_k^h = \nabla^\top \bs{q}_k^h$ and for the dual maximizer $\bs{q}_{k-1}^h$ of the last iteration we implement  \eqref{eq:breg_uk_lifted} as:
\begin{flalign}
&\bs{u}^h_k = \arg\min_{\bs{u}^h:\Omega^h \mapsto \R^l} \max_{\bs{q}^h_k:\Omega^h \mapsto\mathcal{K}} z, \\
&z := \sum_{x\in\Omega^h} (\bs{\rho}^h)^{**}(x,\bs{u}^h(x)) + \langle \bs{q}^h_k - \bs{q}^h_{k-1}, \nabla \bs{u}^h \rangle .
\end{flalign}

\subsection{Transformation of Subgradients}
In Prop.~\ref{prop:equivalency} we formulated a constraint on the subgradients for which the original and lifted Bregman iteration are equivalent. While this property is not necessarily satisfied if the subgradient $\bs{p}^h_{k-1}$ is chosen according to the previous paragraph, we will now show that any such subgradient can be transformed into another valid subgradient that satisfies condition \eqref{eq:lifted_grad},  in analogy to the construction of the subgradient $\bar{p}$ in Sect.~\ref{sss:subgradients_fulilling}.

Consider a pointwise sublabel-integral solution $\bs{u}^h_k$ with subgradient $\bs{p}^h_k := \nabla^\top \bs{q}^h_k \in \partial \bs{J}^h(\bs{u}^h_k)$ for $\bs{q}^h_k(\cdot) \in \bs{K}$ being a maximizer of \eqref{eq:q}. We define a pointwise transformation: For fixed $x^m \in \Omega^h$ and $\bs{u}^h_k(x^m) = \bs{1_i^\alpha}$, let $(\bs{q}^h_k(x^m))^i \in \R^d$ denote the $i$-th row of $\bs{q}^h_k(x^m)$ corresponding to the $i$-th label as prescribed by $\bs{u}^h_k(x^m) = \bs{1}_i^\alpha$. Both in the isotropic and anisotropic case the transformation
\bfl 
        \tilde{\bs{q}}^h_k(x^m) := \frac{(\bs{q}^h_k(x^m))^i}{\gamma_{i+1} - \gamma_i} \tilde{\bs{\gamma}}
        \label{eq:gradient_transform}
\efl
returns an element of the set $\bs{K}$, i.e., $\bs{K}_{\text{iso}}$ or $\bs{K}_{\text{an}}$. In the anisotropic case we can furthermore show that $\tilde{\bs{q}}^h_k$ also maximizes \eqref{eq:q} and therefore the transformation gives a subgradient $\tilde{\bs{p}}^h_k := \nabla^\top \tilde{ \bs{q}}^h_k \in \partial \bs{J}^h(\bs{u}^h_k)$ of the desired form \eqref{eq:lifted_grad}. The restriction to the anisotropic case is unfortunate but necessary due to the fact that the coarea formula does not hold in the discretized case for the usual isotropic discretizations.

\begin{proposition} 
Consider the anisotropic TV-regularized case \eqref{eq:K_an1}-\eqref{eq:K_an2}. Assume that the  iterate $\bs{u}^h_{k}$ is sublabel-integral. Moreover, assume that $\bs{p}^h_k := \nabla^\top \bs{q}^h_k$ is a subgradient in $\partial \bs{J}^h(\bs{u}^h_{k})$ and define $\tilde{\bs{q}}^h_k$ pointwise as in \eqref{eq:gradient_transform}. Then $\tilde{\bs{p}}^h_k := \nabla^\top \tilde{ \bs{q}}^h_k$ is also a subgradient and furthermore of the form
\bfl 
\tilde{\bs{p}}^h_k \quad = \quad p^h_{k} \tilde{\bs{\gamma}}^h,\label{eq:ptildapkm1}
\efl
where $p^h_{k}$ is a subgradient in the unlifted case, i.e., $p^h_{k}\in\partial J^h(u^h_{k})$.
\label{prop:gradient_an}
\end{proposition}

\begin{proof}[Proof of Proposition \ref{prop:gradient_an}]
 The proof consists of two parts. First, we show that the transformation \eqref{eq:gradient_transform}  of any subgradient $\bs{p}_k^h$ in the lifted setting leads to another valid subgradient $\tilde{\bs{p}}_k^h$ in the lifted setting of the form \eqref{eq:ptildapkm1}. Second, we show that the prefactor  $p^h_{k}$ in \eqref{eq:ptildapkm1} is a valid subgradient in the unlifted setting.

In the anisotropic case the spatial dimensions are uncoupled, therefore w.l.o.g. assume $d=1$. Consider two neighboring points $x^m$ and $x^{m+1}$ with $\bs{u}^h_k(x^m) = \bs{1}_i^\alpha$ and $\bs{u}^h_k(x^{m+1}) = \bs{1}_j^\beta$. Applying the forward difference operator, we have $h \nabla \bs{u}^h_k(x^m)=$ 
\begin{flalign}
\begin{cases}
(\bs{0}_{i-1}, \, 1-\alpha, \, \bs{1}_{j-i-2}, \, \beta, \, \bs{0}_{l-j})^\top, & \quad  i < j, \\
(\bs{0}_{i-1}, \, \beta - \alpha, \, \bs{0}_{l-i})^\top, & \quad i = j, \\
(\bs{0}_{j-1}, \, \beta -1, \, -\bs{1}_{i-j-2}, \, -\alpha, \, \bs{0}_{l-j})^\top, & \quad  i > j.
\end{cases}
\end{flalign} 
Maximizers $\bs{q}^h_k(x^m) \in \bs{K}_{\text{an}}$ of the dual problem \eqref{eq:q} are exactly all vectors  of the form $\bs{q}^h_k(x^m)=$
\begin{flalign}
\begin{cases}
(**, \, \gamma_{i+1}-\gamma_i, \, ... , \, \gamma_{j+1} - \gamma_j, \, **)^\top, & \quad i < j, \\
(**, \, \sgn(\beta - \alpha) (\gamma_{i+1}-\gamma_i), \, **)^\top, & \quad i = j, \\
(**, \, \gamma_{j} - \gamma_{j+1} , \, ... , \, \gamma_i - \gamma_{i+1}, \, **)^\top, & \quad i > j. 
\label{eq:Kan_minimizers}
\end{cases}
\end{flalign}
The elements marked with $*$ can be chosen arbitrarily as long as $\bs{q}^h_k(x^m)\in\bs{K}_{\text{an}}$. Due to this special form, the transformation \eqref{eq:gradient_transform} leads to $ \tilde{\bs{q}}^h_k(x^m) = \pm \tilde{\bs{\gamma}}$  depending on the case. Crucially, this transformed vector is another equally valid choice in~\eqref{eq:Kan_minimizers} and therefore \eqref{eq:gradient_transform} returns another valid subgradient $\tilde{\bs{p}}^h_k = \nabla^\top \tilde{\bs{q}}^h_k$.

In order to show that $p_k^h= \nabla^\top q_k^h$ for $q_k^h(\cdot)= \pm 1$ is a subgradient in the unlifted setting we use the same arguments. To this end, we use  the sublabel-accurate notation with $L=2$. The ``lifted'' label space is  $\bs{\Gamma} = [0,1]$, independently of the actual $\Gamma \subset \R$; see \cite[Prop.~3]{sublabel_cvpr}. Then with $u^h_k(x^m) = \gamma_i^\alpha$ and $u^h_k(x^{m+1}) = \gamma_j^\beta$ (corresponding to $\bs{1}_i^\alpha$ and $\bs{1}_j^\beta$ from before), applying the forward difference operator $\nabla u^h_k(x^m) = \frac{1}{h} ( \gamma_j^\beta - \gamma_i^\alpha )$ shows that dual maximizers are $q^h_k(x^m) = \sgn(\gamma_j^\beta - \gamma_i^\alpha) | \bs{\Gamma} | = \pm 1$. It can be seen that the algebraic signs coincide pointwise in the lifted and unlifted setting.  Thus $p^h_{k}$ in \eqref{eq:ptildapkm1} is of the form $p^h_k = \nabla^\top q_k^h$ and in particular a subgradient in the unlifted setting. \qed
\end{proof}

%% file: Sections/figure_rof.tex
\newcommand{\circimg}[1]{\includegraphics[width=.16\linewidth]{#1}}
\begin{figure*}[!htb]
  \begin{center}
  \begin{tabular}{lllll}
    \circimg{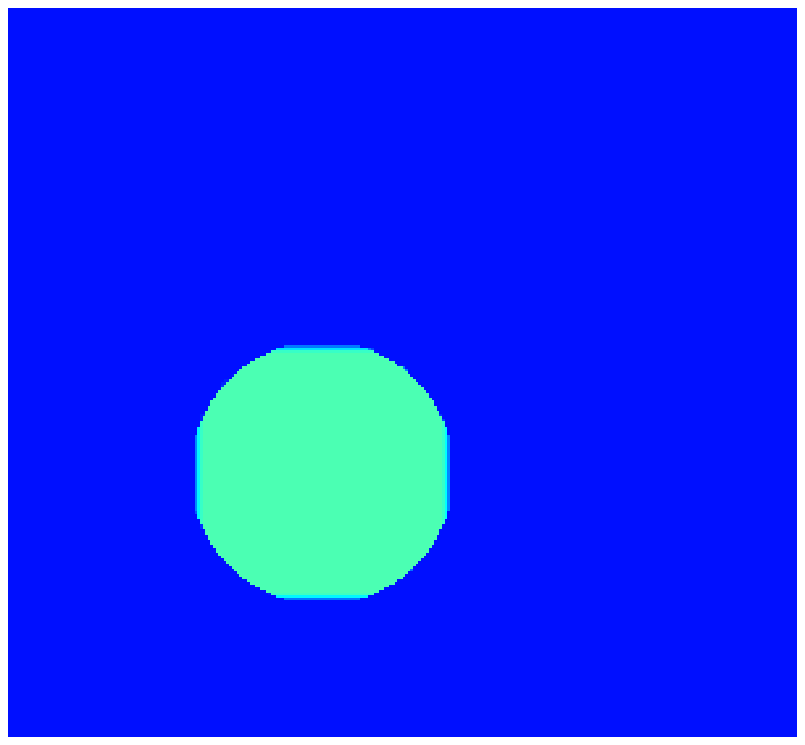}&
    \circimg{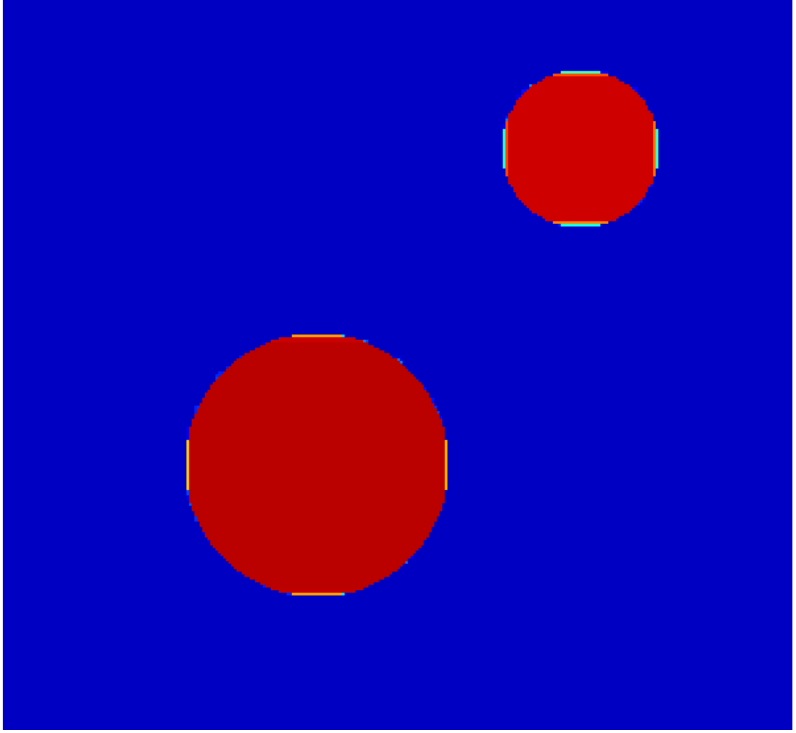}&
    \circimg{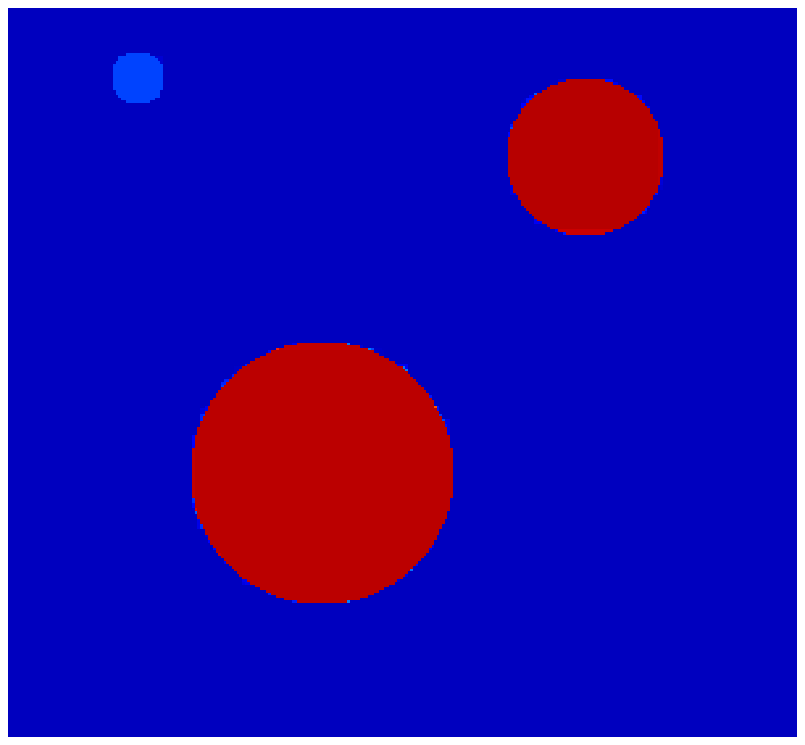}&
    \circimg{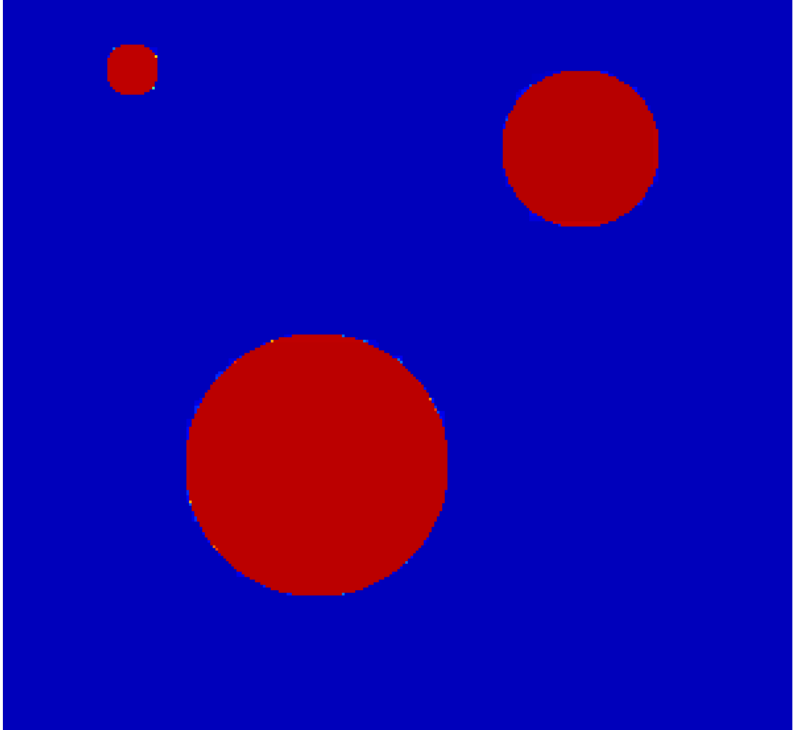}&
    \circimg{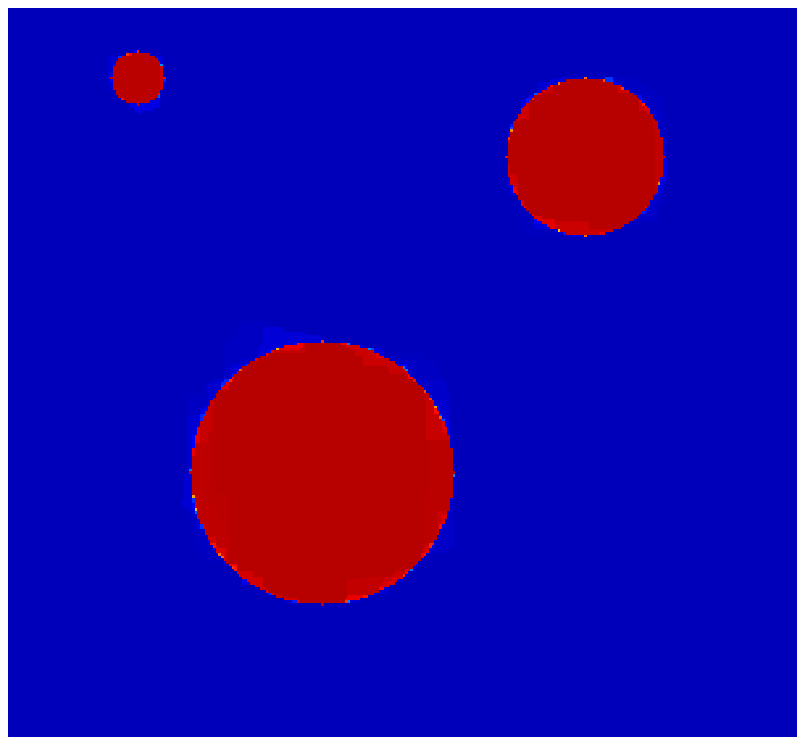}\\
    \circimg{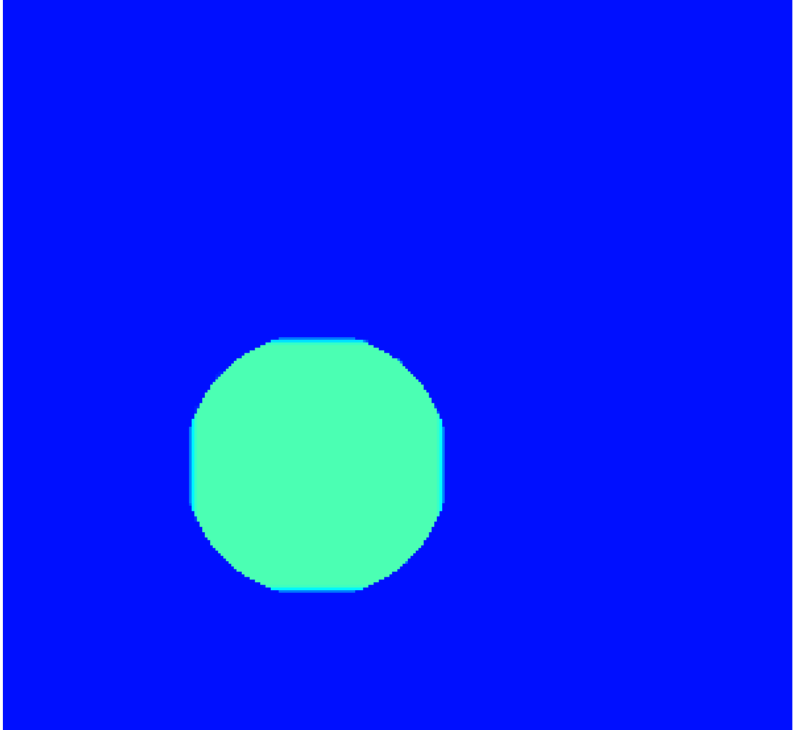}&
    \circimg{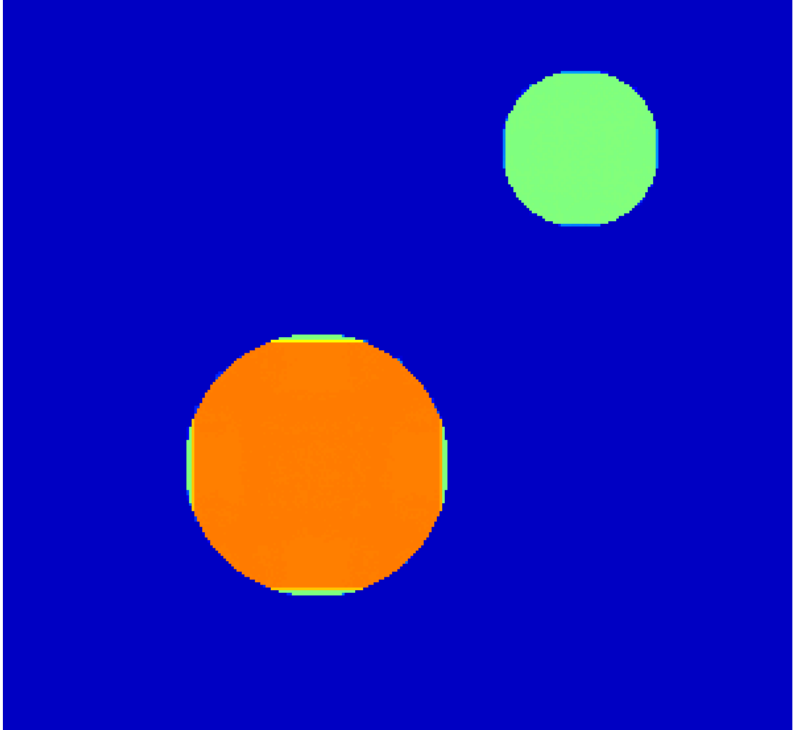}&
    \circimg{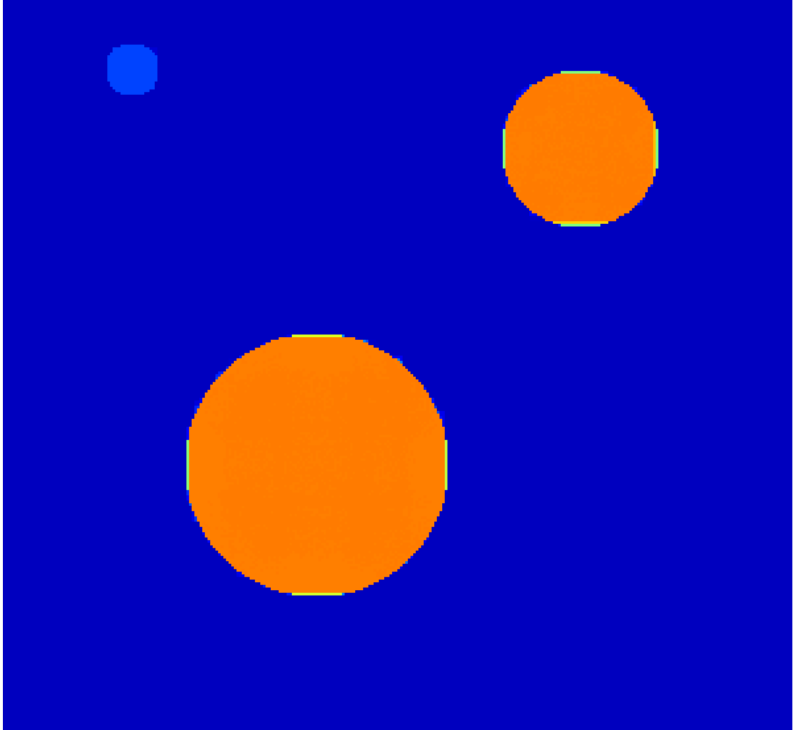}&
    \circimg{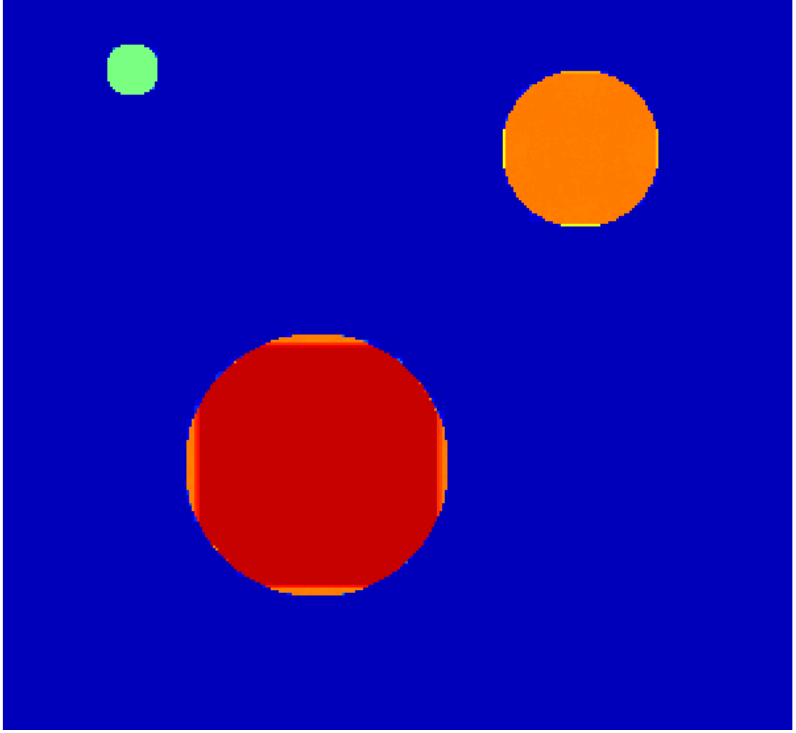}&
    \circimg{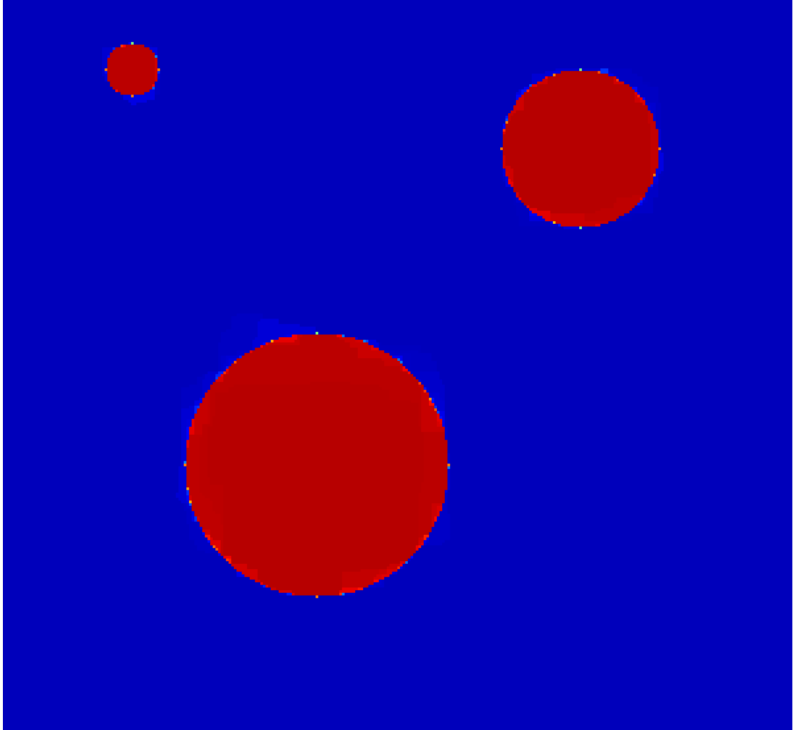}\\
    \circimg{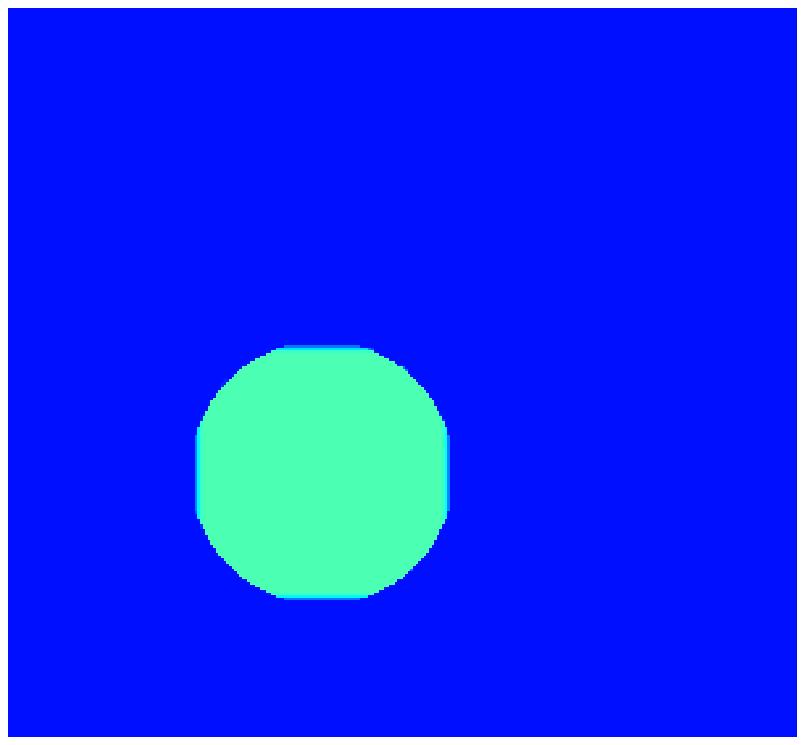}&
    \circimg{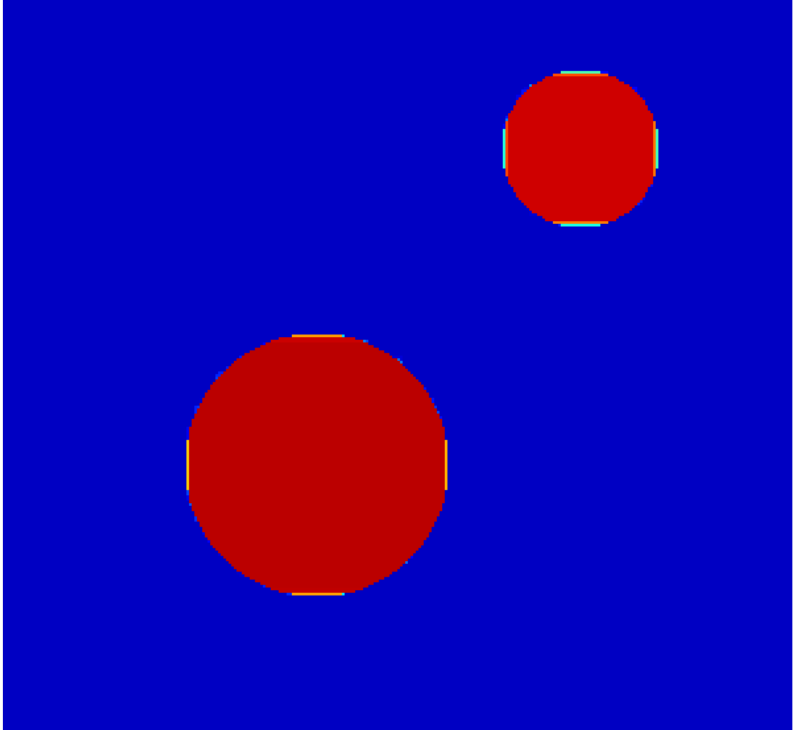}&
    \circimg{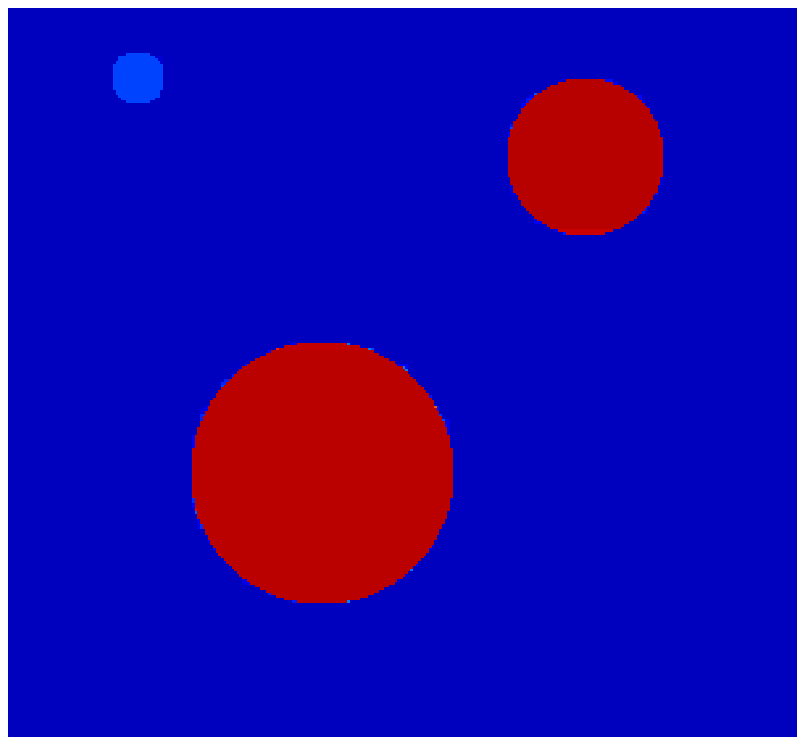}&
    \circimg{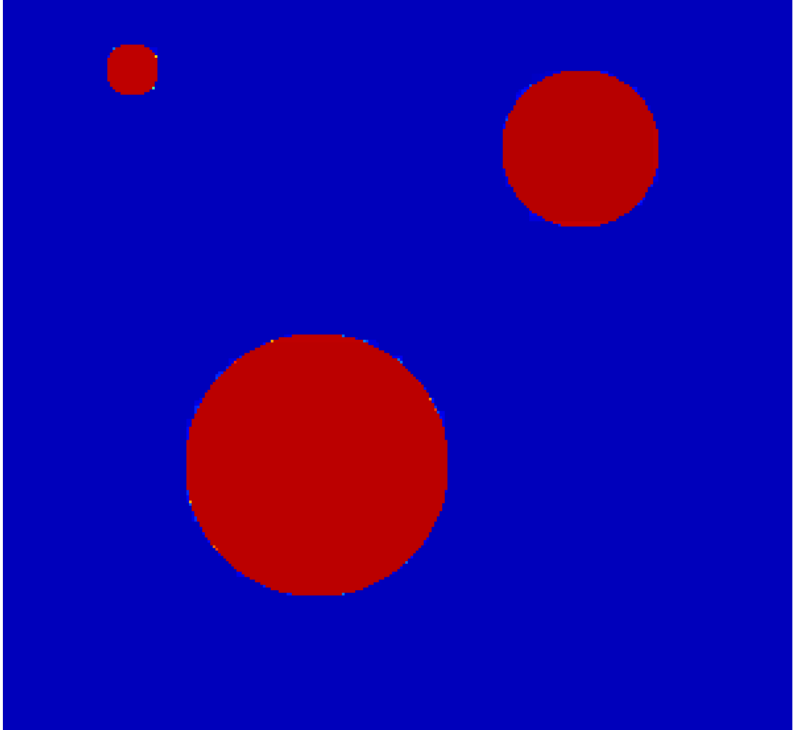}&
    \includegraphics[width=.24\linewidth]{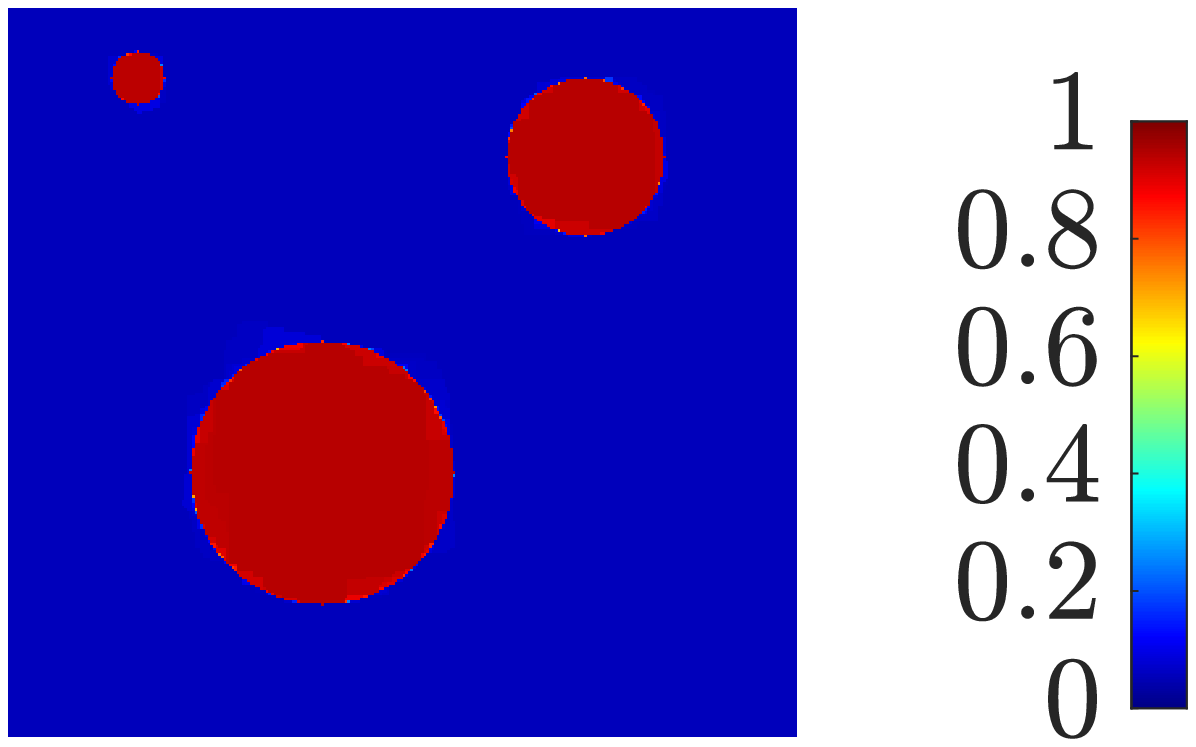}\\
    $\qquad k=1$&
    $\qquad k=2$&
    $\qquad k=3$&
    $\qquad k=4$&
    $\qquad k=50$
  \end{tabular}%
  \end{center}%
\vspace{-1em}\caption{\textbf{Equivalence of classical and lifted Bregman on a convex problem.} On the \emph{convex} ROF-\eqref{eq:problem} problem with anisotropic TV, a plain implementation of the classical Bregman iteration as in Alg.~1 \textbf{(top row)} and a na\"ive implementation of the lifted generalization as in Alg.~2 with lifted subgradients taken unmodified from the previous iteration \textbf{(middle row)} show clear differences. If the lifted subgradients are first transformed as in Prop.~\ref{prop:equivalency}, the lifted iterates \textbf{(bottom row)} are visually indistinguishable from the classical iteration in this fully convex case. In addition, the lifted version also allows to transparently handle non-convex energies (Fig.~\ref{fig:results2}).
%verison  two approaches Comparison of Alg.~1 and Alg.~2. Top row shows the result of the original Bregman iteration, the two bottom rows the results of the lifted Bregman iteration. In the third row we use transformed subgradients. In the latter case  and Alg.~2 are equivalent and return the same results up to negligible numerical differences; i.e. the 2-norm of the difference is of the order 1.0e-02 -- 1.0e-03
}
\label{fig:rof}
\end{figure*}

%% file: Sections/numerical_results.tex
\section{Numerical Results}
\label{sec:numerical_results}

In this section, we investigate the equivalence of the original and lifted Bregman iteration for~the ROF-\eqref{eq:problem} problem numerically. Furthermore, we present a stereo-matching example which supports our conjecture that the lifted Bregman iteration for variational models with arbitrary data terms can be used to decompose solutions into eigenfunctions of the regularizer.  

\subsection{Convex Energy with Synthetic Data}
We compare the results of the original and lifted Bregman iteration for the ROF-\eqref{eq:problem} problem with $\lambda =20$, synthetic input data and anisotropic TV regularizer. In the lifted setting, we compare implementations with and without transforming the subgradients as in~\eqref{eq:gradient_transform}. The results shown in Fig.~\ref{fig:rof} clearly support the theory: Once subgradients are transformed as in Prop.~\ref{prop:equivalency}, the iterates agree with the classical, unlifted iteration.

A subtle issue concerns points where the minimizer of the lifted energy is not sublabel-integral, i.e., cannot be easily identified with a solution of the original problem. This impedes the recovery of a suitable subgradient as in \eqref{eq:projection}, which leads to diverging Bregman iterations. We found this issue to occur in particular with isotropic TV discretization, which does not satisfy a discrete version of the coarea formula -- which is used to prove in the continuous setting that solutions of the original problem can be recovered by thresholding -- but is also visible to a smaller extent around the boundaries of the objects in Fig.~\ref{fig:rof}, especially when subgradients are not modified.

\subsection{Non-Convex Stereo Matching with Artificial Data} 

\input{Sections/figure_rainbow.tex}

In the following two toy examples, we empirically investigate how properties of the Bregman iteration carry over to the lifted Bregman iteration for \emph{arbitrary} (non-convex) data terms. We consider a relatively simple stereo-matching problem, namely TV-\eqref{eq:problem} with the non-convex data term
\bfl
        \rho(x,u(x)) = h_\tau\left( | I_1(x_1, x_2 + u(x)) - I_2(x) |\right) .
        \label{eq:stereo_matching_data_term_simple} 
\efl
Here, $I_1$ and $I_2$ are two given input images and $h_\tau(\alpha) := \min \ls \tau, \alpha \rs$ is a threshholding function. We assume that the input images are rectified, i.e., the epipolar lines in the images align, so that the unknown -- but desired -- displacement of points between the two images is restricted to the $x_2$ axis and can be modeled as a scalar function $u$.

A typical observation when using nonlinear scale space method is that components in the solution corresponding to non-linear eigenfunctions of the regularizer appear at certains points in time depending on their eigenvalue.

We thus construct $I_1$ and $I_2$ such that the solution $\tilde{u}$ of $\arg\min_u \rho(x,u(x))$ is clearly the sum of eigenfunctions of the isotropic [anisotropic] $\TV$, i.e., multiples of indicator functions of circles [squares].

 In the following we elaborate the isotropic setting. For some non-overlapping circles $B_{r_i}(m_i)$ with centers $m_i$ and radii $r_i$ we would like the solution
\bfl
        \tilde{u}(x) &= \sum_i \mathbbm{1}_{B_{r_i}(m_i)}(x) = \begin{cases}
                1, &\text{within } B_{r_i}(m_i), \\
                0, &\text{elsewhere}.
        \end{cases}
\efl
Fig.~\ref{fig:rainbow} shows the corresponding data; note that there is no displacement except inside the circles, where it is non-zero but constant.

In analogy to the convex ROF example in Fig.~\ref{fig:rof} and the theory of inverse scale space flow, we would expect the following property to hold for the lifted Bregman iteration:
The solutions -- here the depth maps of the artificial scene -- returned in each iteration of the lifted Bregman iteration progressively incorporate the discs (eigenfunctions of isotropic TV) according to their radius (associated eigenvalue); the biggest disc should appear first, the smallest disc last.

Encouragingly, these expectations are also observed in this non-convex case, see Fig.~\ref{fig:rainbow_results}. This suggests that the lifted Bregman iteration could be useful to decompose the solution of a variational problem with arbitrary data term with respect to eigenfunctions of the regularizer.

\input{Sections/figure_rainbow_result.tex}

\subsection{Non-Convex Stereo Matching with Real-World Data}

We also computed results 
for a stereo-matching problem with real life data.   We used TV-\eqref{eq:problem} and the  data term \cite{sublabel_cvpr}
\bfl
\rho(x,u(x)) =   \int_{W(x)} \sum_{j=1,2} h_\tau(d_j(I_1(y), I_2(y))) \text{d}y .
\label{eq:stereo_matching_data_term}
\efl
Here, $W(x)$ denotes a patch around $x$, $h_\tau$ is the truncation function with threshhold $\tau$ and $d_j$ is the absolute gradient difference
\bfl
d_j(I_1(y), I_2(y)) = | \partial_{x_j} I_1(y_1,y_2+u(x)) - \partial_{x_j} I_2(y_1,y_2)|.
\efl
 This data term is non-convex and non-linear in~$u$. We applied the lifted Bregman iteration on three data sets \cite{data_middlebury_bike} using $L=5$ labels, the isotropic TV regularizer and untransformed subgradients. The results can be seen in Fig.~\ref{fig:stereo_profile} (Motorbike: $\lambda =20$, $k=30$) and Fig.~\ref{fig:results2} (Umbrella: $\lambda = 10$; Backpack: $\lambda = 25$). We also ran the experiment with an anisotropic TV regularizer as well as transformed subgradients. Overall, the behavior was similar, but transforming the subgradients led to more pronounced jumps compared to the isotropic case.

Again, the evolution of the depth map throughout the iteration is reminiscent of an inverse scale space flow. The first solution is a smooth approximation of the depth proportions and as the iteration continues, finer structures are added. This behaviour is also visible in the progression of the horizontal profiles depicted in Fig.~\ref{fig:stereo_profile}.

\input{Sections/figure_stereo.tex}

%% file: Sections/figure_rainbow.tex
\renewcommand{\circimg}[1]{\includegraphics[width=.61\linewidth]{#1}}
\newcommand{\circimgs}[1]{\includegraphics[width=.30\linewidth]{#1}}

\begin{figure}[t]
  \begin{center}
  \begin{tabular}{l l}
    \circimg{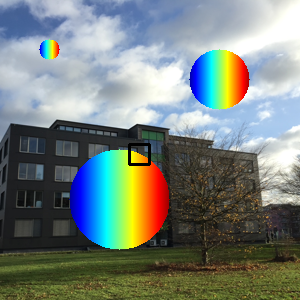}&
    \shortstack{\circimgs{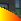}  \\
    \circimgs{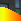}}
  \end{tabular}%
  \end{center}%
\vspace{-1em}\caption{ \textbf{Artificial data for stereo matching.} The backgrounds of the input images $I_1$ and $I_2$ are chosen to be identical. Only within the three circles $I_1$ and $I_2$ differ. The information within the circles is shifted four pixels sideways; the circles themselves stay in place. The black square marks the area of the close-ups of $I_1$ \textbf{(top)} and $I_2$ \textbf{(bottom)}. }
\label{fig:rainbow}
\end{figure}

%% file: Sections/figure_rainbow_result.tex
\newcommand{\circimgrb}[1]{\includegraphics[width=.23\linewidth]{#1}}
\newcommand{\circimgrbsmall}[1]{\includegraphics[width=.15\linewidth]{#1}}
\begin{figure*}
  \begin{center}
  \begin{tabular}{cccc}
    \circimgrb{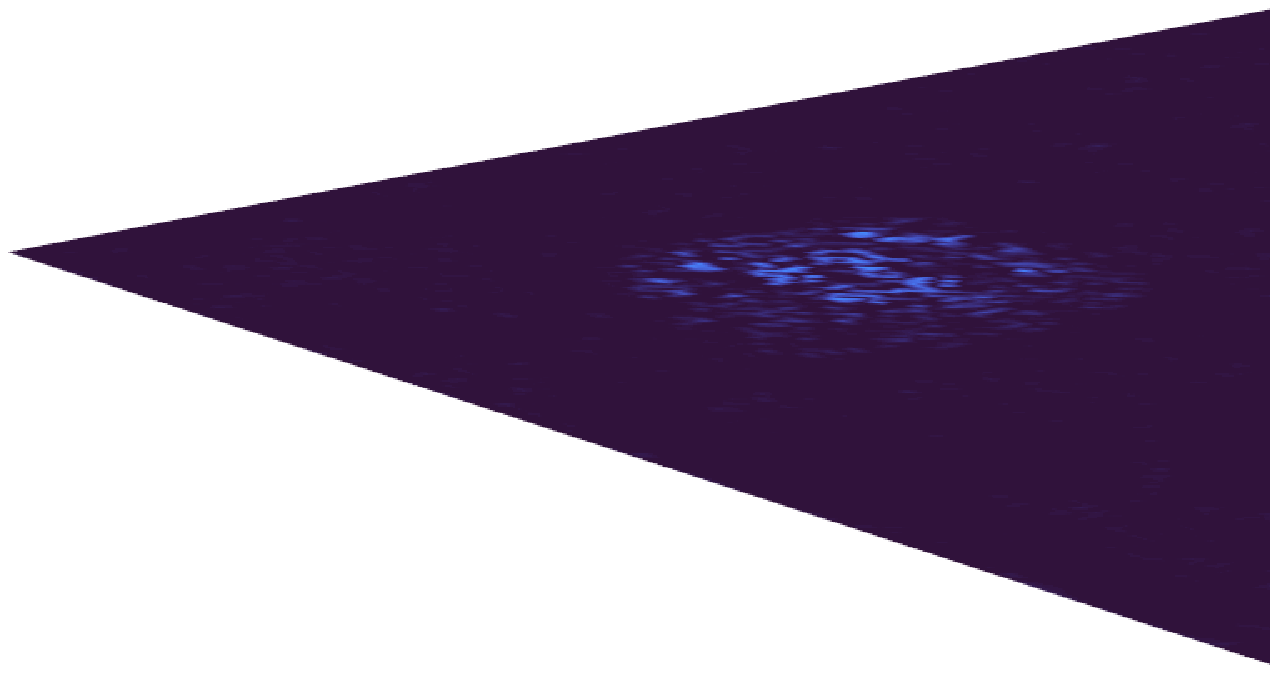}&
    \circimgrb{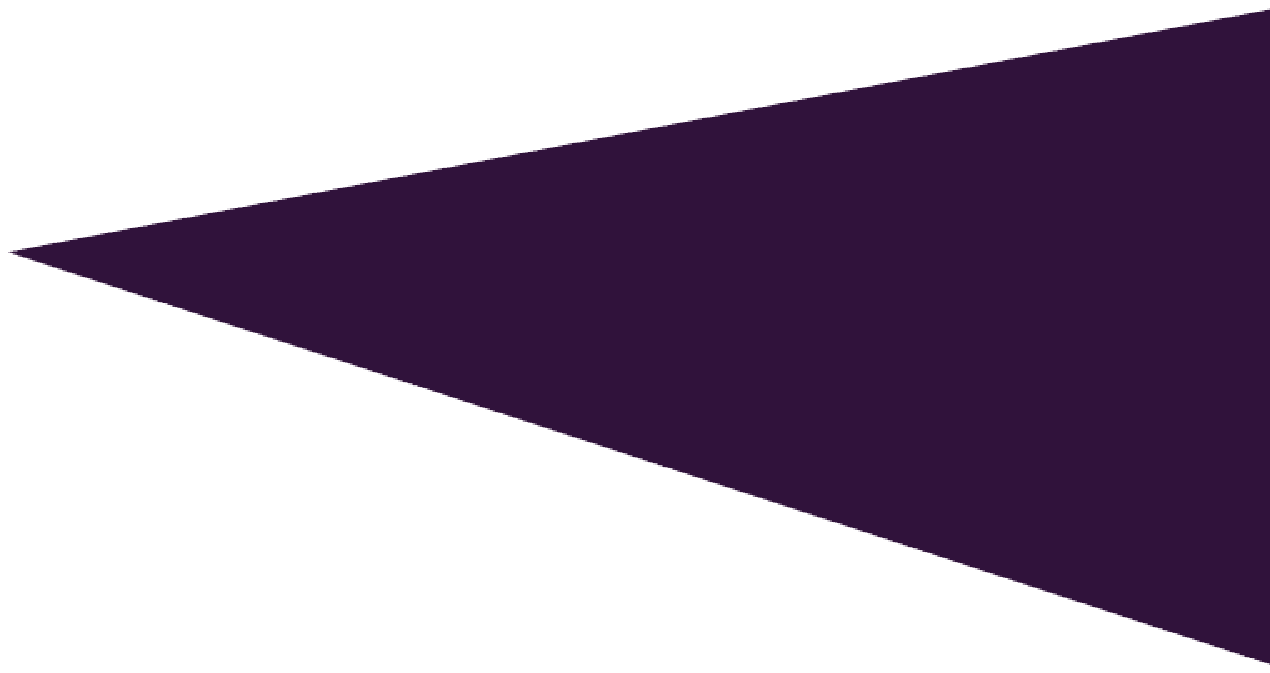}&
    \circimgrb{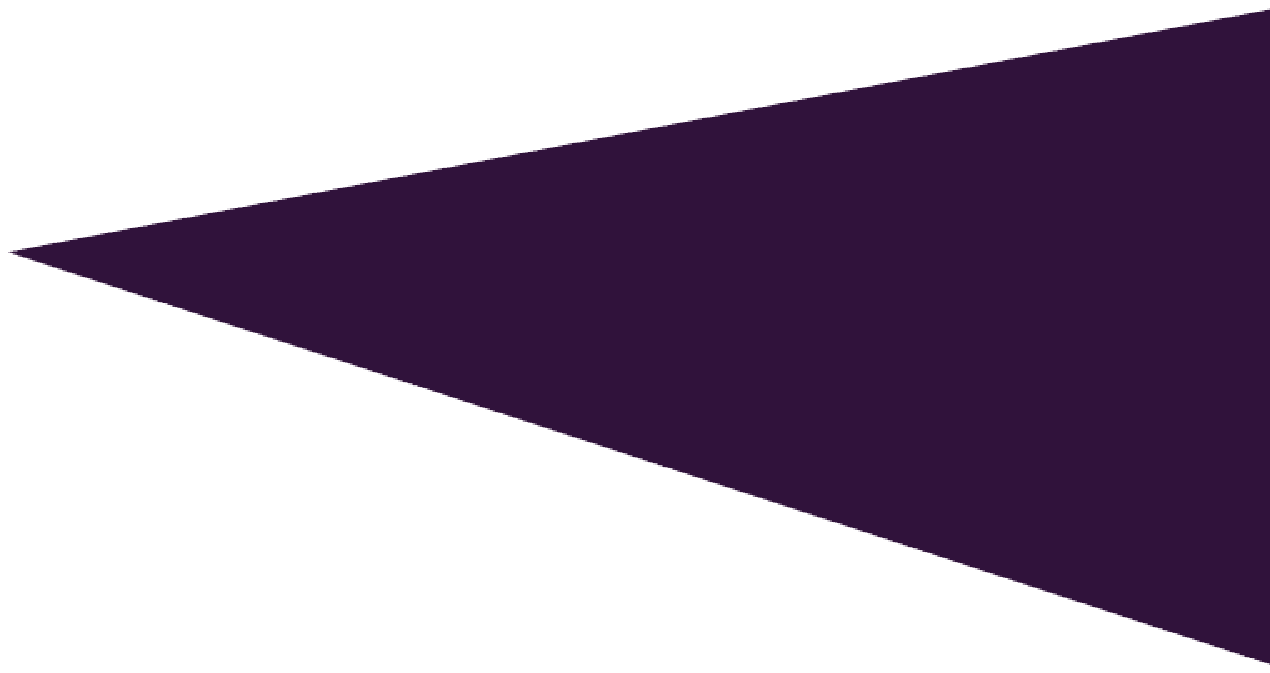}&
    \circimgrb{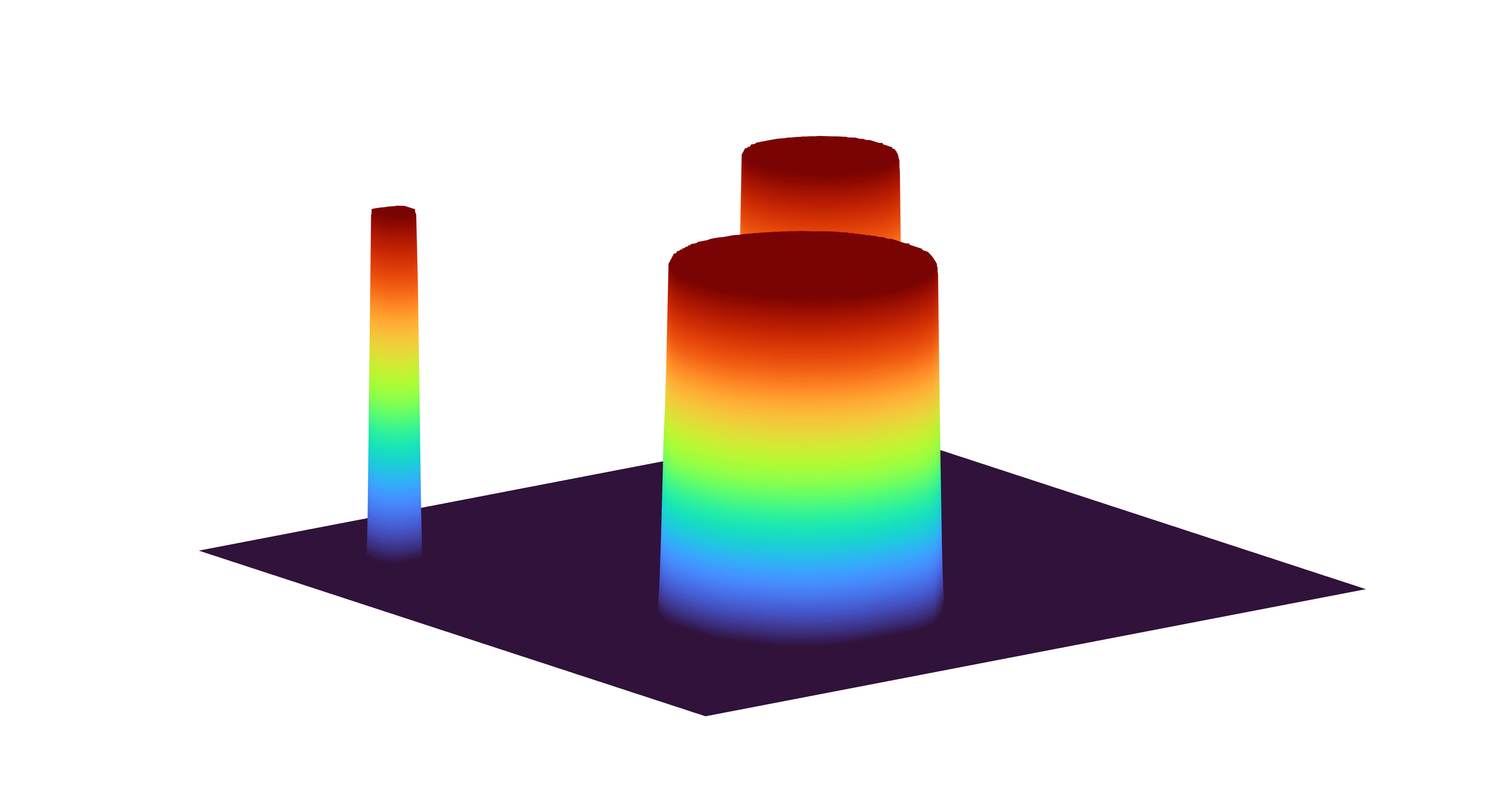}
      \\
    $k=1$ & $k=2$ & $ k=4$ & $ k=10$  \\
    \\
    \circimgrbsmall{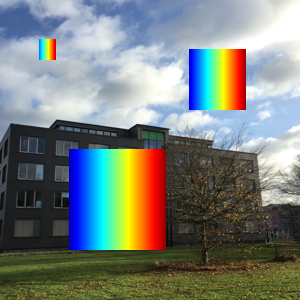}&
    \circimgrb{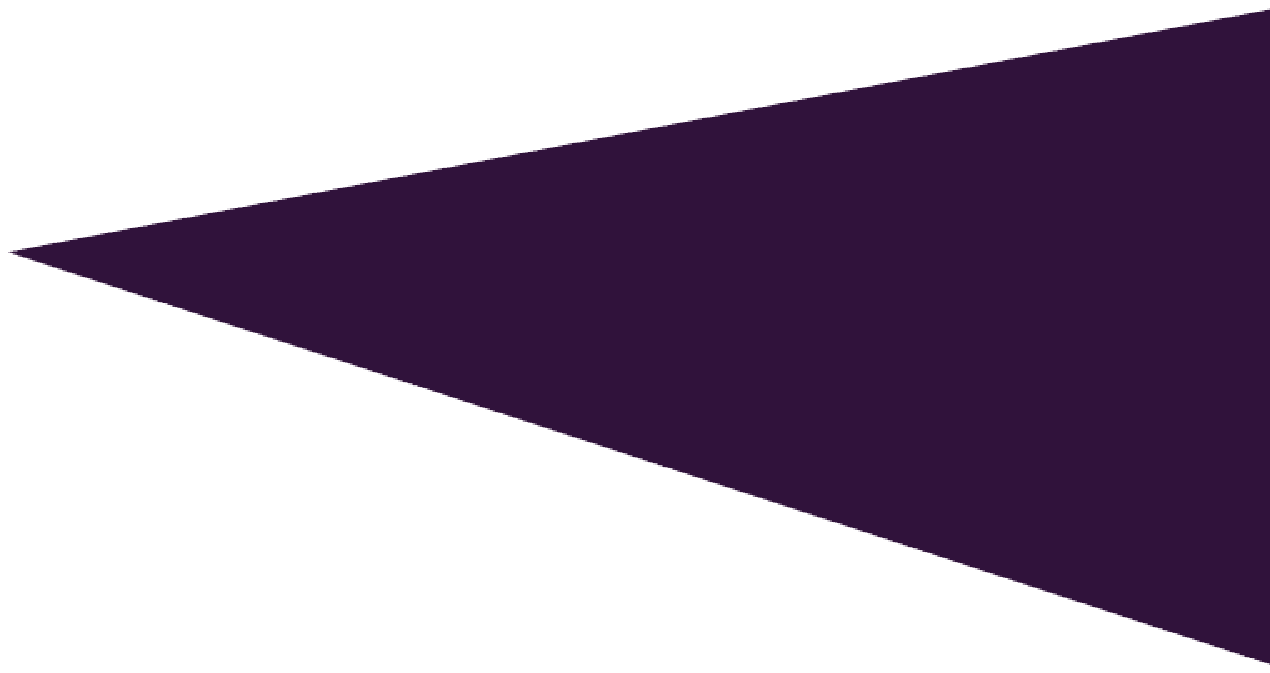}&
    \circimgrb{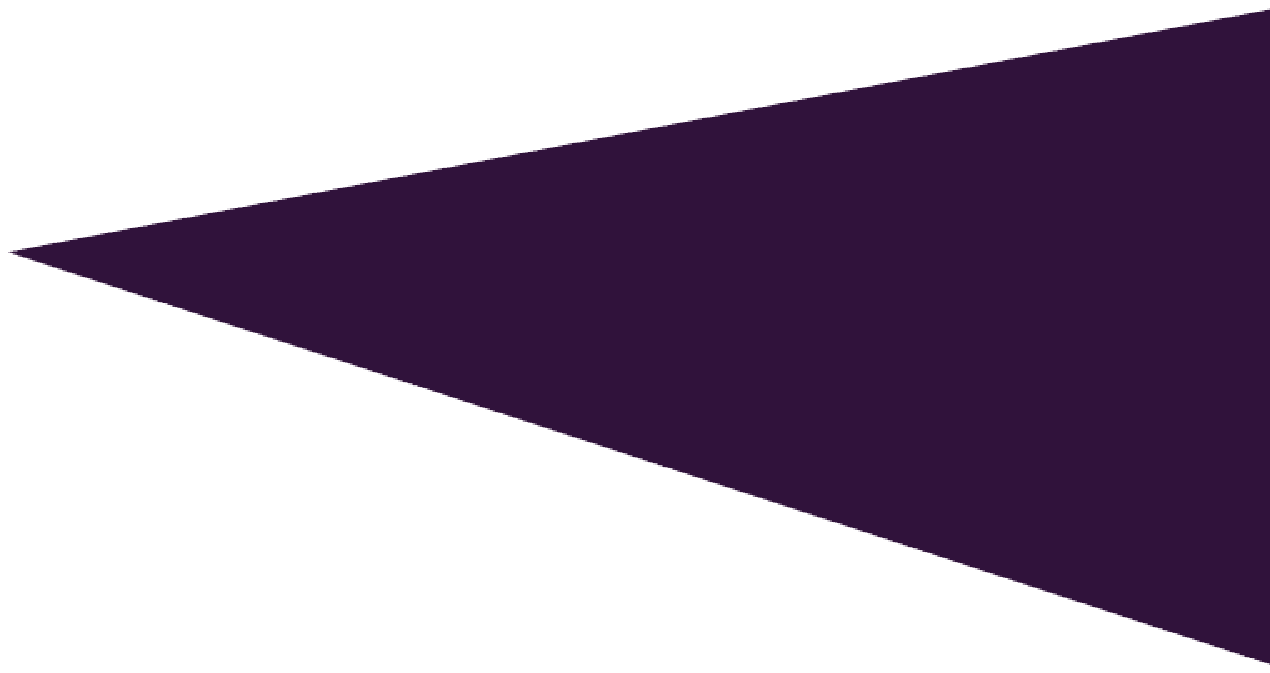}&
    \circimgrb{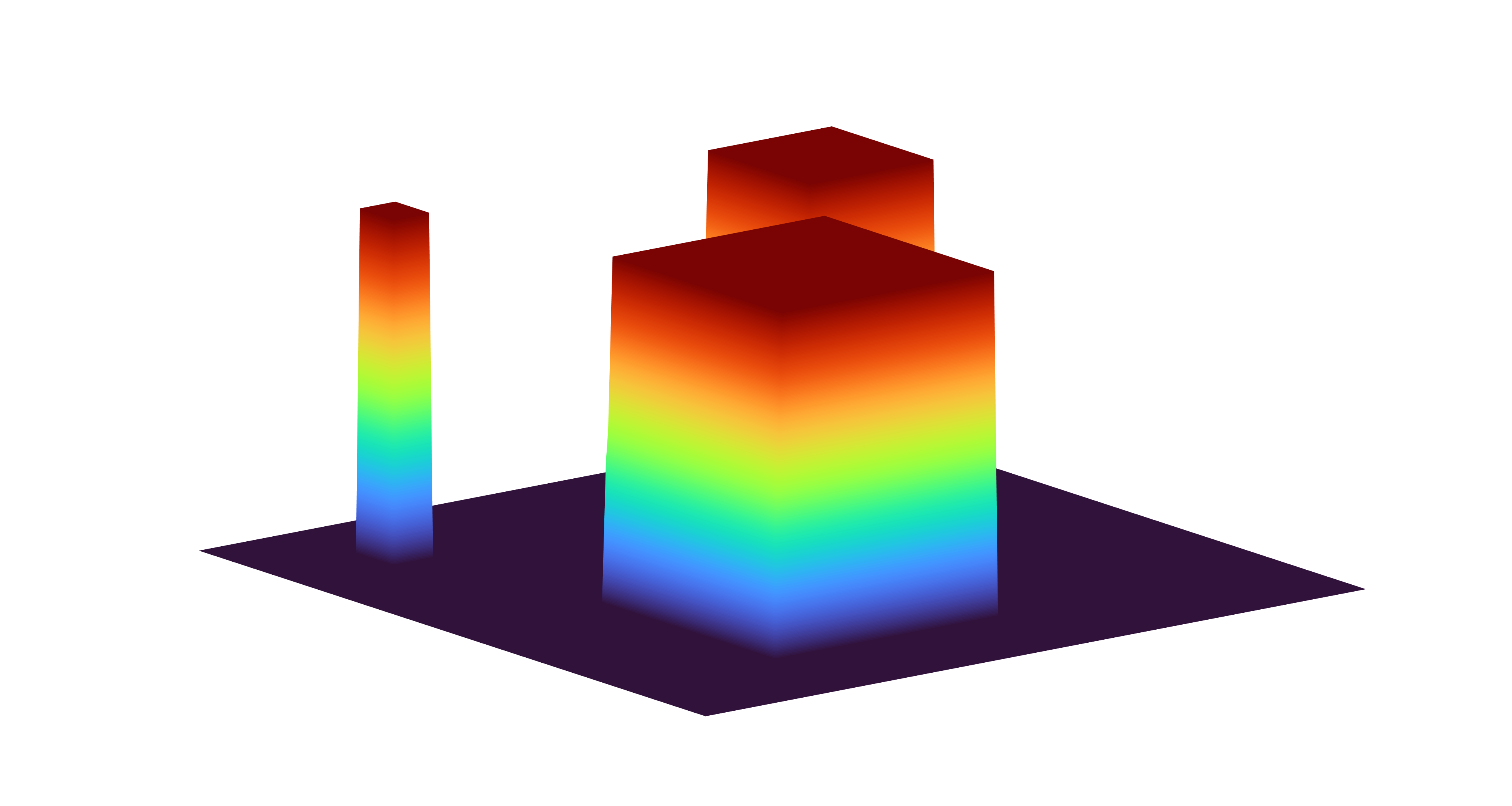}
      \\
    Input & $k=1$ & $ k=3$ & $ k=6$ 
  \end{tabular}%
  \end{center}%
\vspace{-1em}
\caption{ \textbf{Lifted Bregman on stereo matching problem with artificial input data.} TV-\eqref{eq:problem} problem with data term \eqref{eq:stereo_matching_data_term_simple}. We use $\lambda = 14 $ and isotropic TV \textbf{(top)}, and $\lambda = 7 $ and anisotropic TV \textbf{(bottom)}, respectively. Fig.~\ref{fig:rainbow} describes our input data in the isotropic setting \textbf{(top)}. In the anisotropic setting we use square cutouts instead \textbf{(bottom left)}. For this non-convex data term, the lifted Bregman iteration progressively adds components corresponding to eigenfunctions of the TV regularizer to the depth map that. Components associated with larger eigenvalues appear earlier.
}
\label{fig:rainbow_results}
\end{figure*}

%% file: Sections/figure_stereo.tex
\newcommand{\circimgst}[1]{\includegraphics[width=.22\linewidth]{#1}}
\begin{figure*}[]
  \begin{center}
  \begin{tabular}{cccc}
    \circimgst{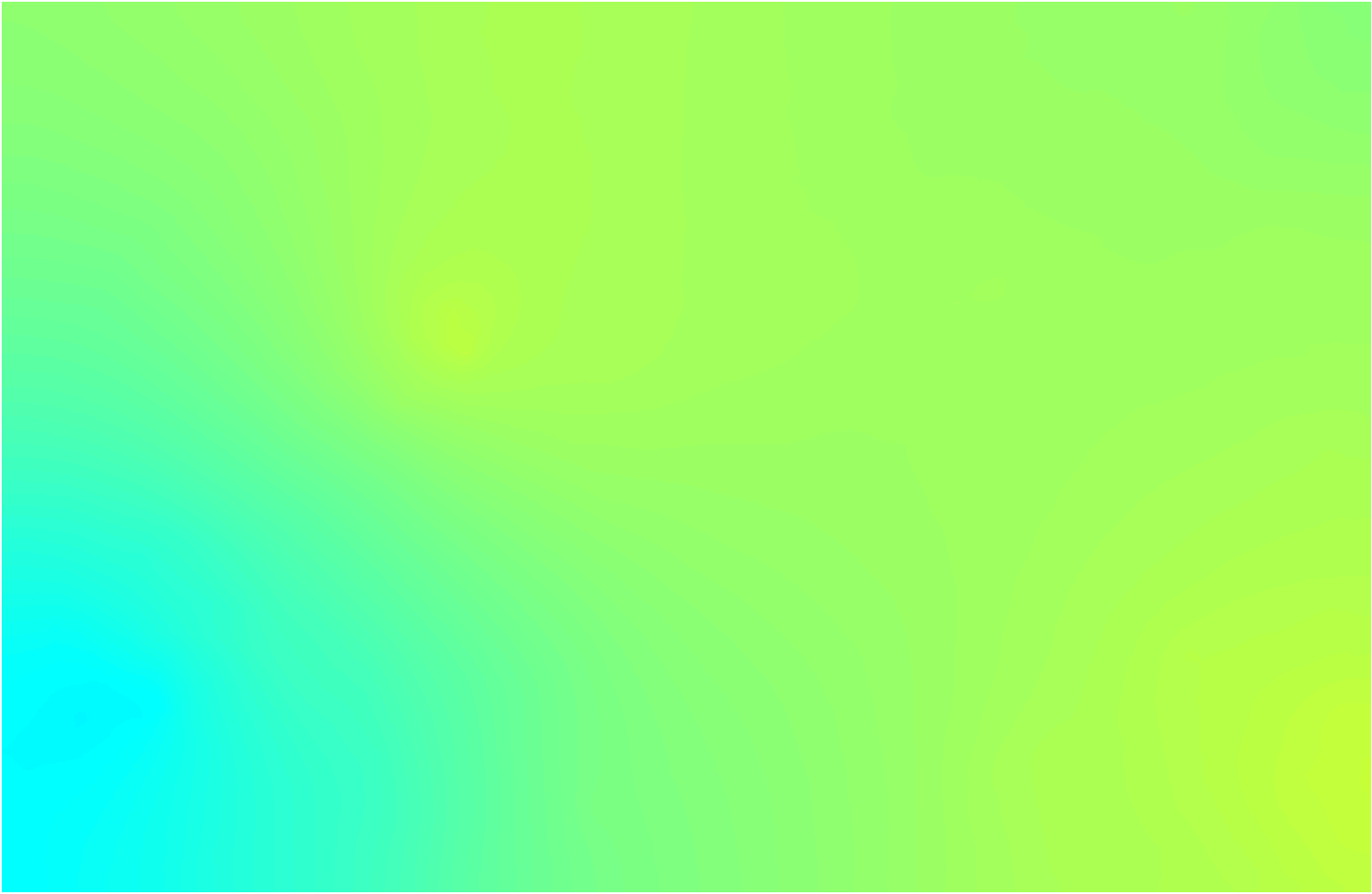}&
    \circimgst{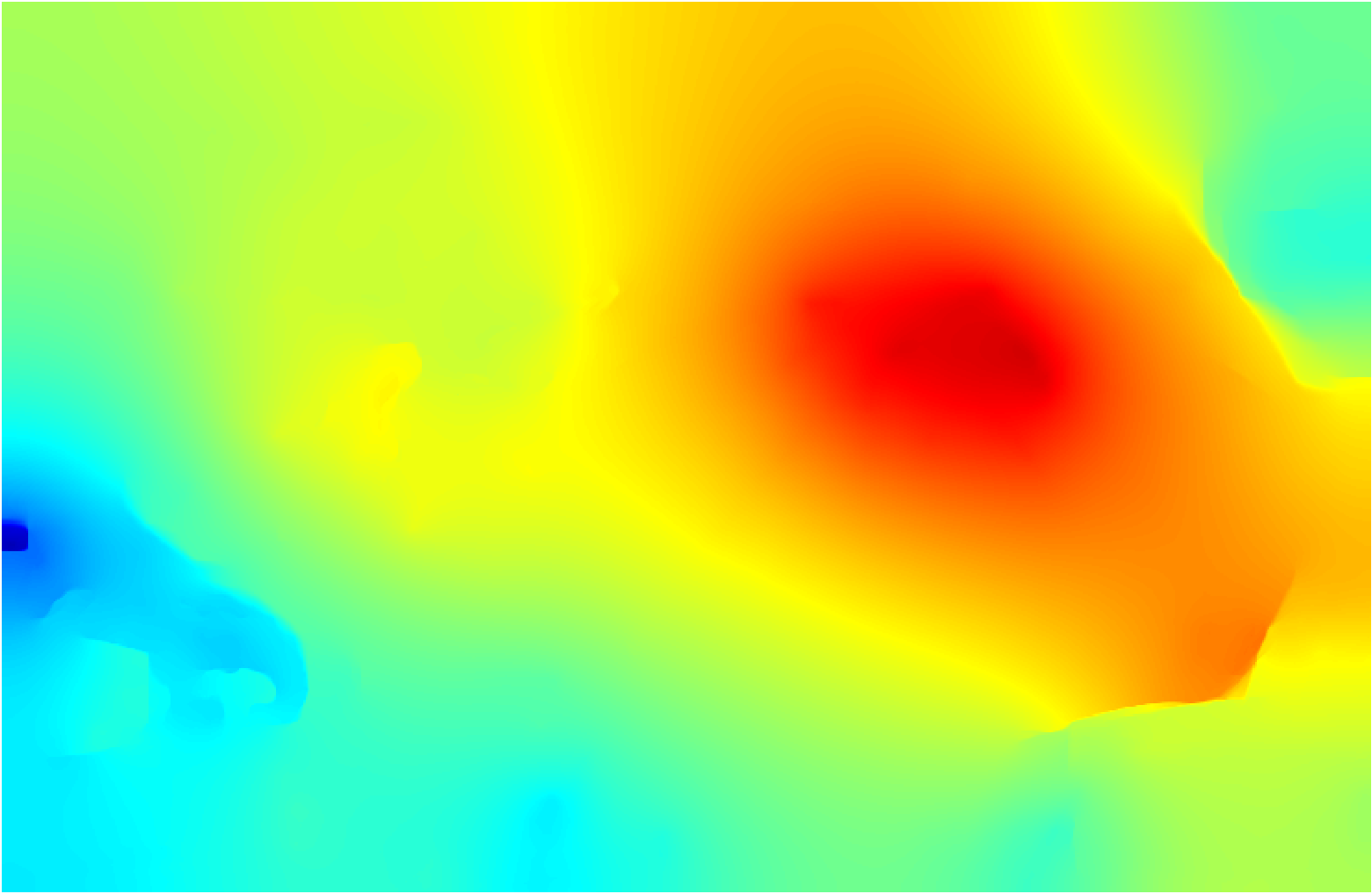}&
    \circimgst{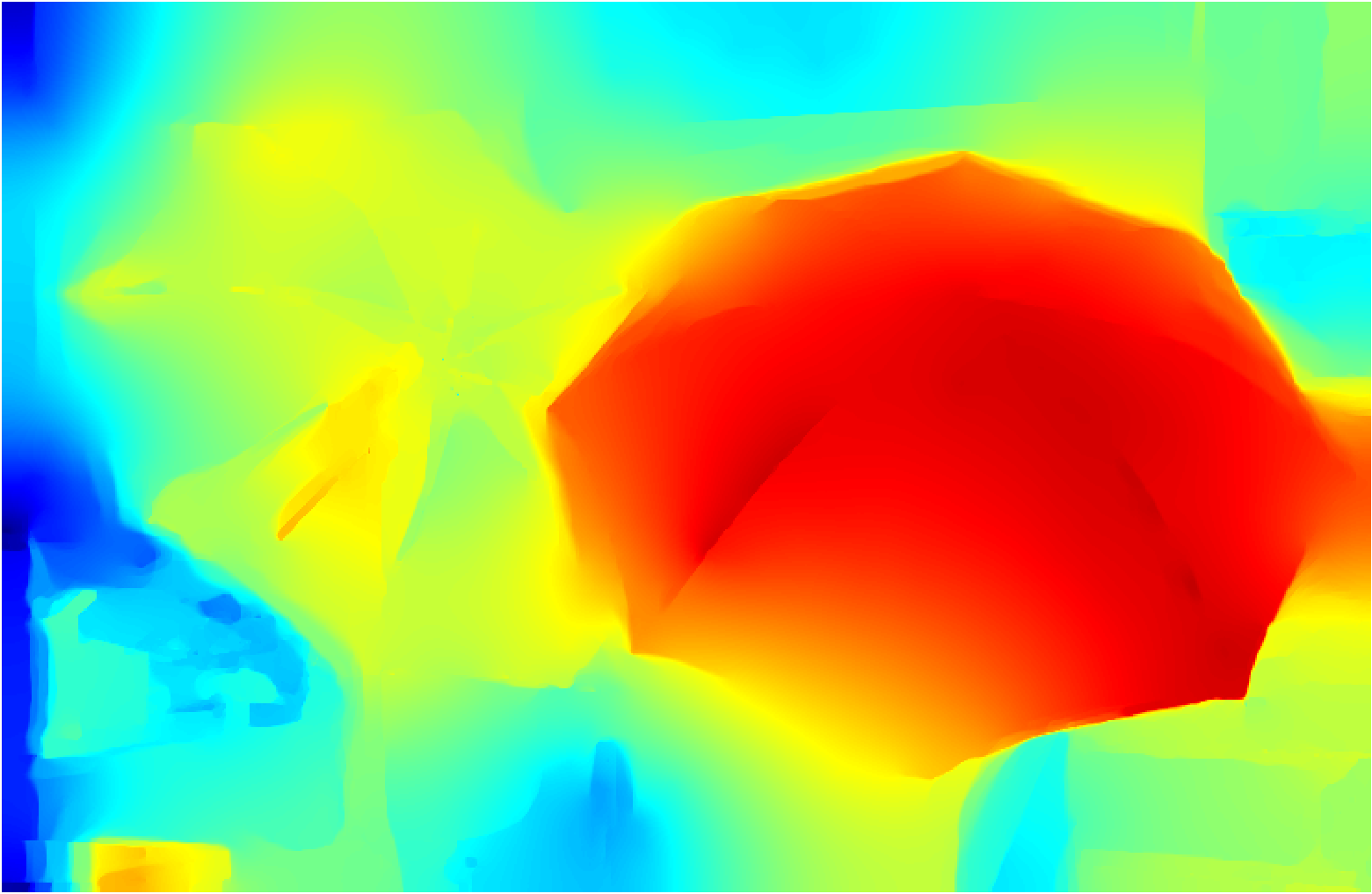}&
    \circimgst{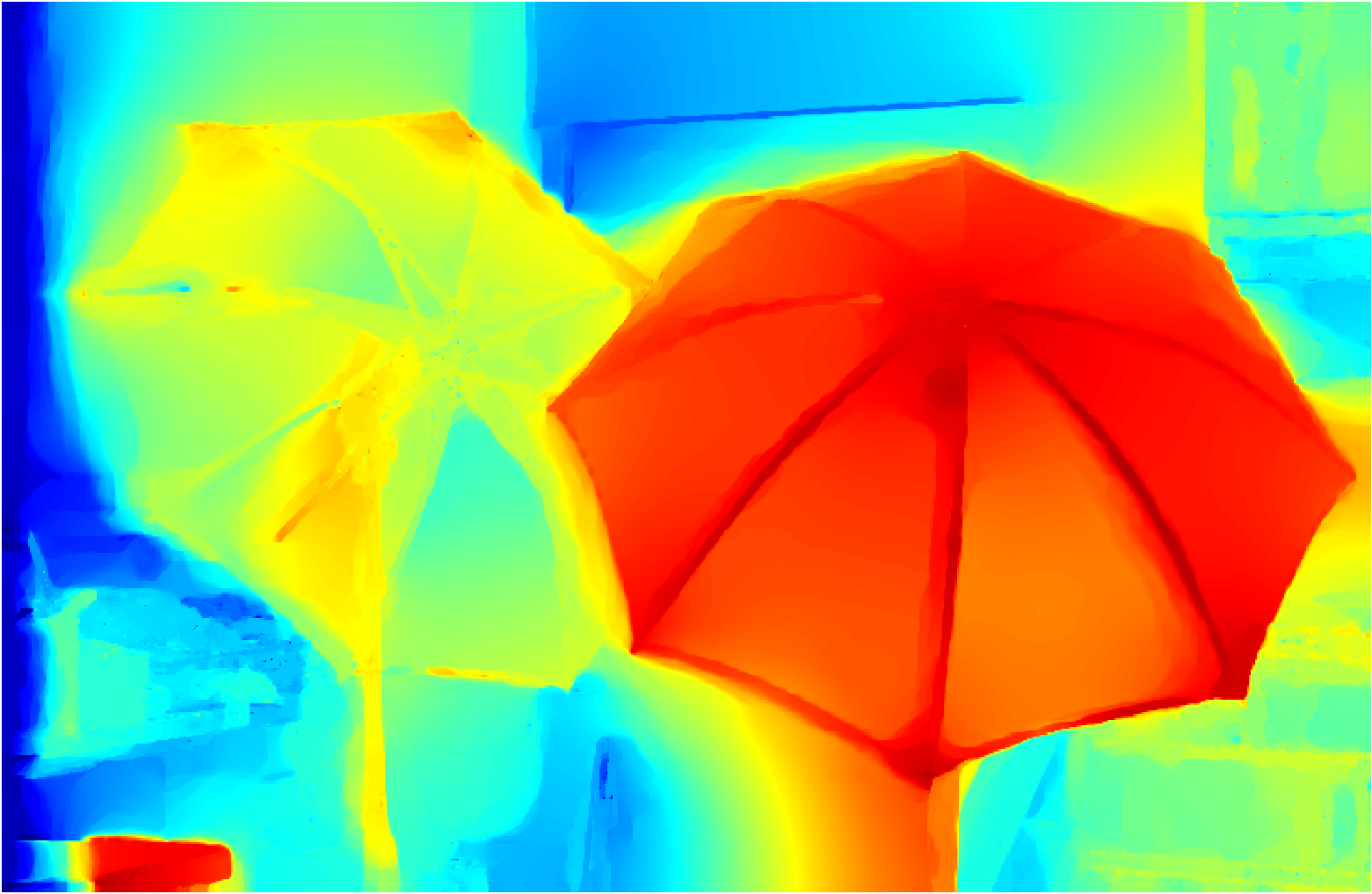} 
     \\   
    \circimgst{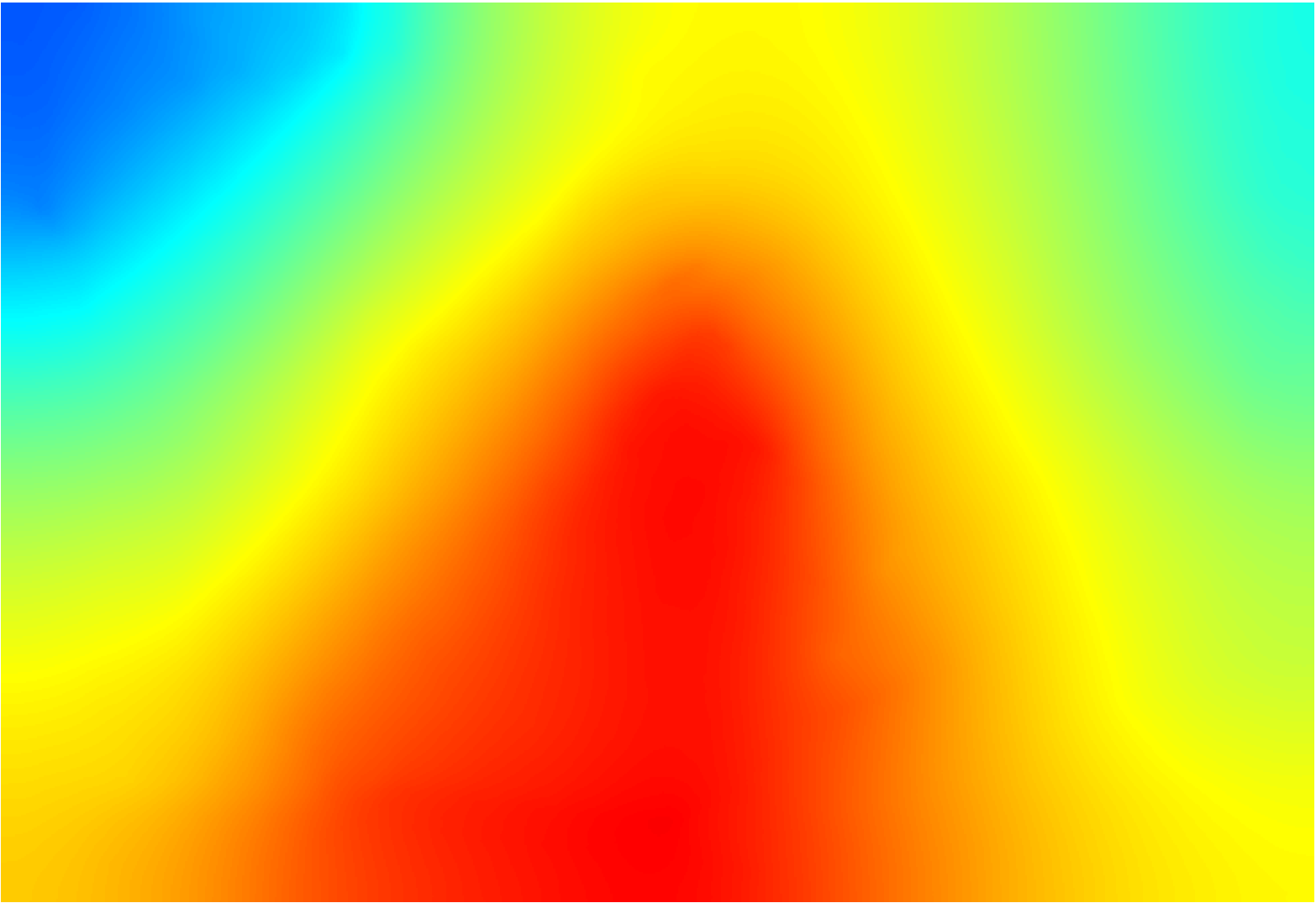}&
    \circimgst{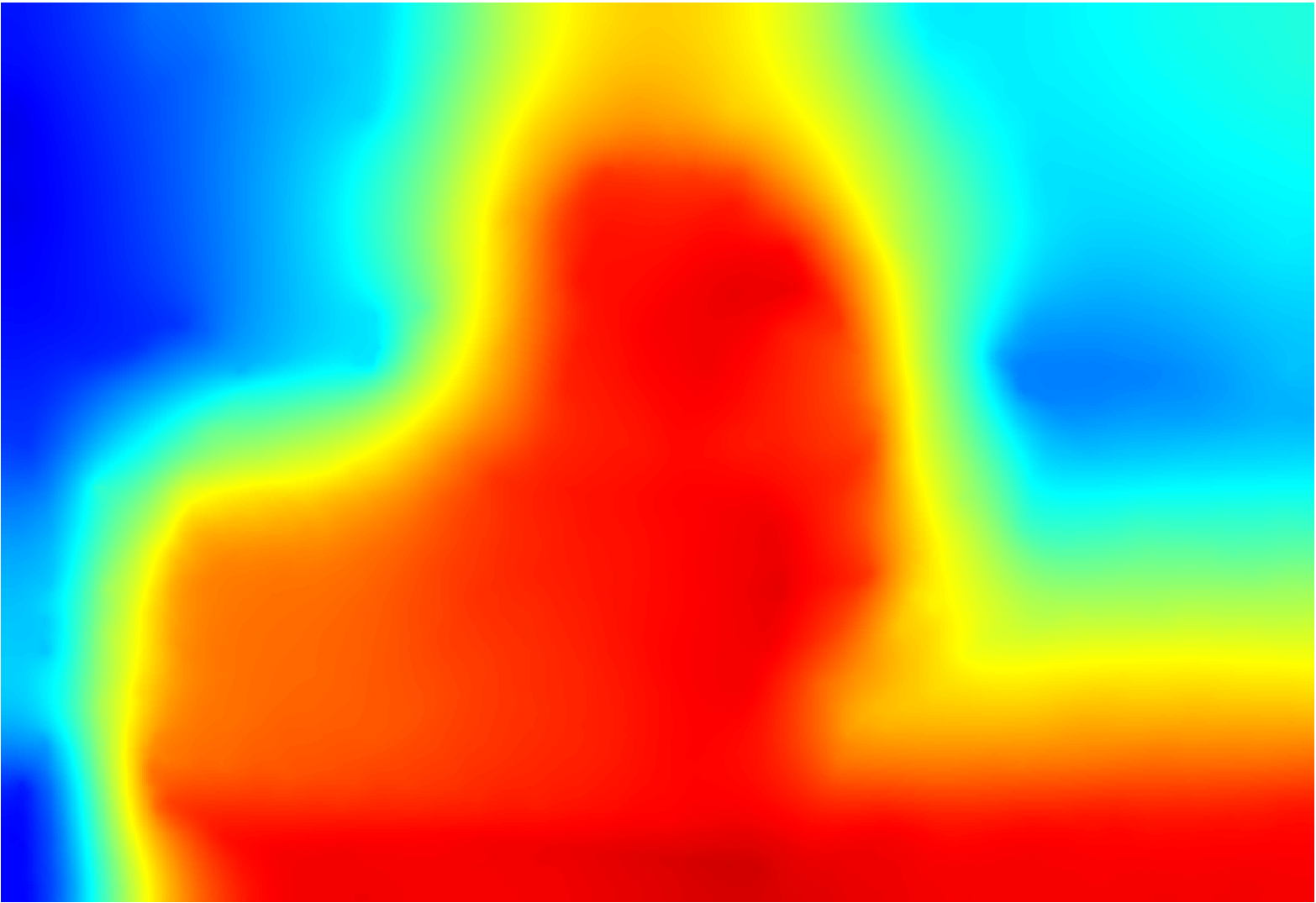}&
    \circimgst{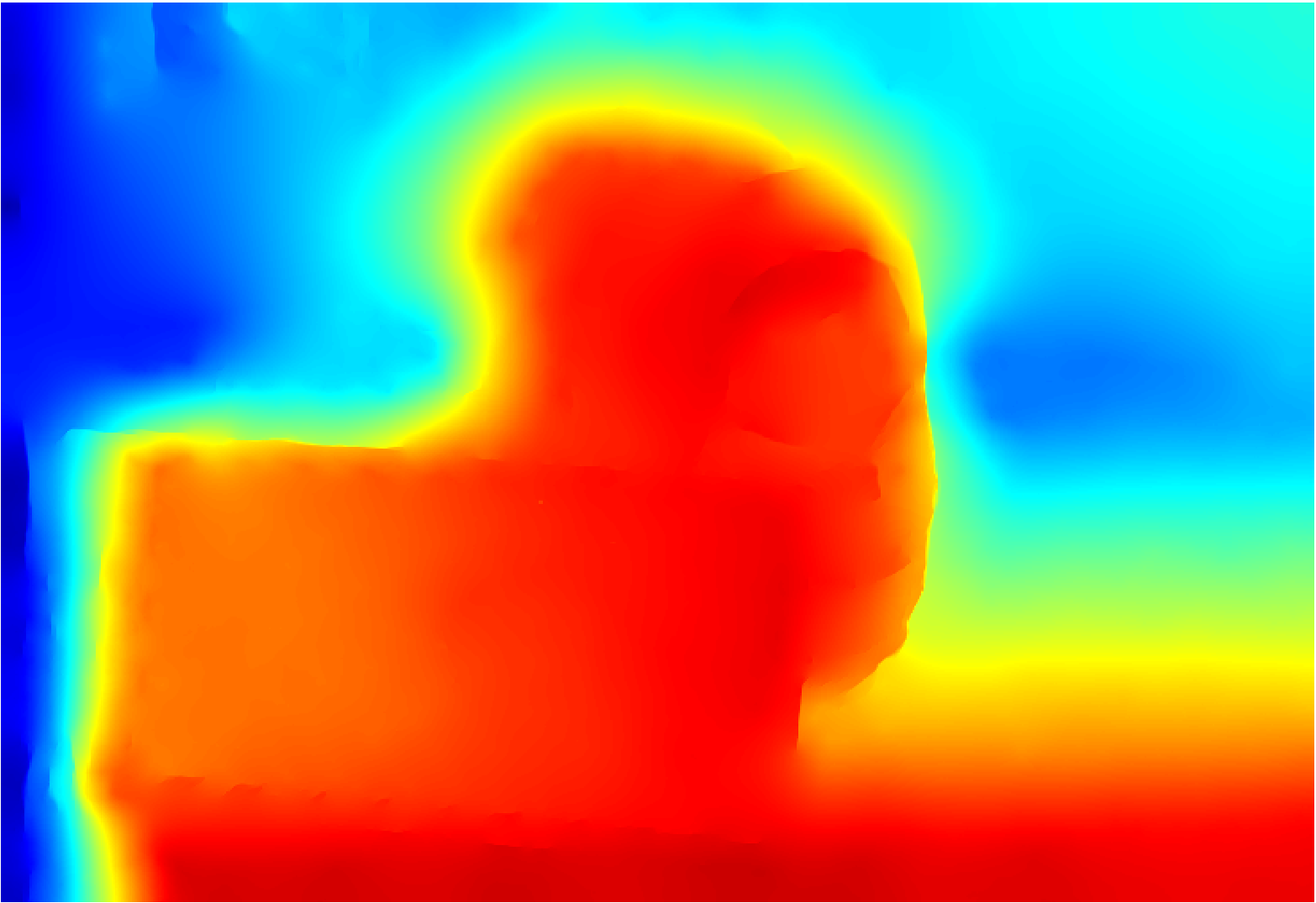}&
    \circimgst{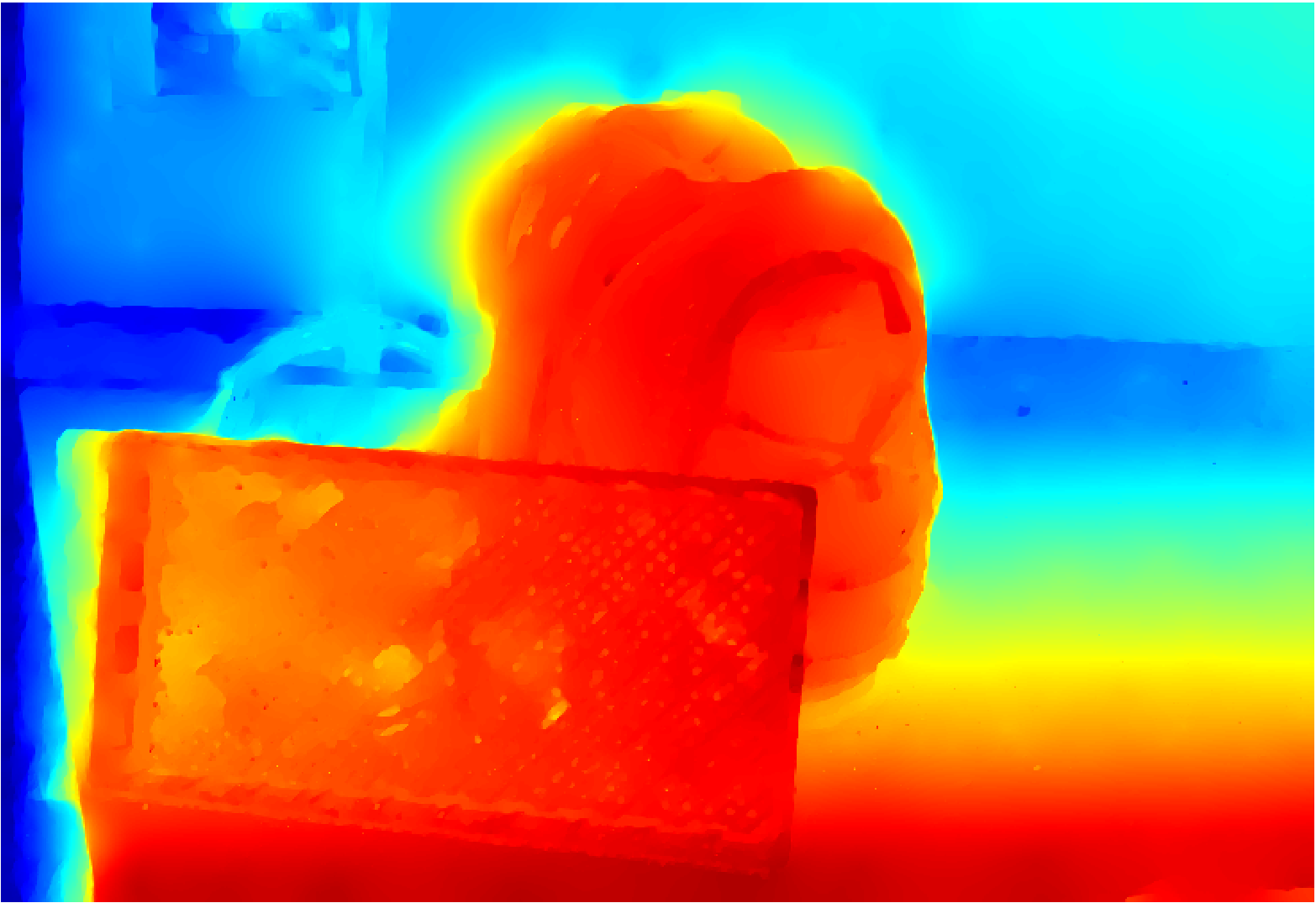} 
    \\
    $ k=1$ & $ k=8$ & $ k=21$ & $ k=60$ 
    \end{tabular}%
  \end{center}%
\vspace{-1em}
\caption{\textbf{Lifted Bregman for non-convex non-linear scale space.} Shown are the Bregman iterations for the non-convex TV-\eqref{eq:problem} stereo matching problem with data term \eqref{eq:stereo_matching_data_term}, evaluated on the Umbrella and Backpack data set from the Middlebury stereo datasets \cite{data_middlebury_bike}.  At $k=1$ the solution is a coarse approximation of the depth field. As the iteration advances, details are progressively recovered according to scale. Although the problem is \emph{not} of the convex and positively one-homogeneous form OH-\eqref{eq:problem} classically associated with the inverse scale space flow, the results show a qualitative similarity to a nonlinear scale space for this difficult non-convex problem. }
\label{fig:results2}
\end{figure*}

%% file: Sections/conclusion.tex
\section{Conclusion}
\label{sec:conclusion}

We have proposed a combination of the Bregman iteration and a lifting approach with sublabel-accurate discretization in order to extend the Bregman iteration to non-convex energies such as stereo matching. If a certain form of the subgradients can be ensured -- which can be shown \dbdetail under some assumptions in the continuous case as well as the discretized case \dbold in particular for total variation regularization -- the iterates agree in theory and in practice for the classical convex ROF-\eqref{eq:problem} problem. \dbreformulated The numerical experiments show behavior in the non-convex case that is very similar to what one expects in classical inverse scale space. This opens up a number of interesting theoretical questions, such as the decomposition into non-linear eigenfunctions, as well as practical applications such as non-convex scale space transformations and nonlinear filters for arbitrary nonconvex data terms. \dbold